\documentclass[11pt,a4paper]{article}
\usepackage[utf8]{inputenc}
\usepackage[english]{babel}
\usepackage[T1]{fontenc}
\usepackage{amsmath,amsfonts,amssymb,amsthm}
\usepackage{lmodern}
\usepackage{multirow}
\usepackage{hyperref}
\usepackage{color}
\usepackage{tikz}
\usepackage{setspace}
\usepackage[left=2cm,right=2cm,top=2cm,bottom=2cm]{geometry}
\usepackage{fancyhdr}  
\title{On the continuity of the tangent cone to the determinantal variety\footnotemark[1]}
\author{Guillaume Olikier\footnotemark[2] \and P.-A. Absil\footnotemark[2]}
\date{November 22, 2021}

\newcommand{\N}{\mathbb{N}}

\newcommand{\R}{\mathbb{R}}

\DeclareMathOperator{\dom}{dom}
\DeclareMathOperator{\im}{im}
\newcommand{\ip}[2]{\langle #1 , #2 \rangle}
\newcommand{\norm}[1]{\Vert #1 \Vert}
\DeclareMathOperator*{\argmin}{argmin}

\DeclareMathOperator*{\inlim}{\underline{Lim}}
\DeclareMathOperator*{\outlim}{\overline{Lim}}
\DeclareMathOperator*{\setlim}{Lim}

\newcommand{\setmapsto}{\multimap}

\newcommand{\gencone}[4]{{#1}_{#2}^{#4}(#3)} 
\newcommand{\tancone}[2]{\gencone{T}{#1}{#2}{}} 
\newcommand{\regtancone}[2]{\gencone{T}{#1}{#2}{\mathrm{C}}} 
\newcommand{\regnorcone}[2]{\gencone{\widehat{N}}{#1}{#2}{}} 
\newcommand{\norcone}[2]{\gencone{N}{#1}{#2}{}} 
\newcommand{\connorcone}[2]{\gencone{N}{#1}{#2}{\mathrm{C}}} 

\newcommand{\oshort}[1]{\mkern 0.9mu\overline{\mkern-0.9mu#1\mkern-0.9mu}\mkern 0.9mu}
\newcommand{\ushort}[1]{\mkern 0.9mu\underline{\mkern-0.9mu#1\mkern-0.9mu}\mkern 0.9mu}

\newcommand{\tp}{\top}
\newcommand{\mpinv}{\dagger}
\newcommand{\st}{\mathrm{St}}
\newcommand{\grass}{\mathrm{Gr}}
\DeclareMathOperator{\rank}{rk}
\DeclareMathOperator{\diag}{diag}

\newtheorem{theorem}{Theorem}[section]
\newtheorem{proposition}{Proposition}[section]
\newtheorem{lemma}{Lemma}[section]
\newtheorem{corollary}{Corollary}[section]

\begin{document}
\maketitle

\thispagestyle{fancy}
\rhead{Tech.\ report UCL-INMA-2021.06-v2}
\lhead{\url{http://sites.uclouvain.be/absil/2021.06}}

\renewcommand{\thefootnote}{\fnsymbol{footnote}}
\footnotetext[1]{This work was supported by the Fonds de la Recherche Scientifique -- FNRS and the Fonds Wetenschappelijk Onderzoek -- Vlaanderen under EOS Project no 30468160.}
\footnotetext[2]{ICTEAM Institute, UCLouvain, Avenue Georges Lema\^{\i}tre 4, 1348 Louvain-la-Neuve, Belgium (\href{mailto:guillaume.olikier@uclouvain.be}{\nolinkurl{guillaume.olikier@uclouvain.be}}, \href{mailto:pa.absil@uclouvain.be}{\nolinkurl{pa.absil@uclouvain.be}}).}
\renewcommand{\thefootnote}{\arabic{footnote}}

\begin{abstract}
Tangent and normal cones play an important role in constrained optimization to describe admissible search directions and, in particular, to formulate optimality conditions.
They notably appear in various recent algorithms for both smooth and nonsmooth low-rank optimization where the feasible set is the set $\R_{\le r}^{m \times n}$ of all $m \times n$ real matrices of rank at most $r$.
In this paper, motivated by the convergence analysis of such algorithms, we study, by computing inner and outer limits, the continuity of the correspondence that maps each $X \in \R_{\le r}^{m \times n}$ to the tangent cone to $\R_{\le r}^{m \times n}$ at $X$. We also deduce results about the continuity of the corresponding normal cone correspondence.
Finally, we show that our results include as a particular case the $a$-regularity of the Whitney stratification of $\R_{\le r}^{m \times n}$ following from the fact that this set is a real algebraic variety, called the real determinantal variety.
\medskip

\medskip

\end{abstract}


\section{Introduction}
\label{sec:Introduction}
In constrained optimization, the feasible set plays a role as important as the objective function: before looking for a descent direction, it is first necessary to know which search directions are admissible. It is now well established that the admissible search directions at a feasible point are described by the so-called tangent and normal cones to the feasible set at that point \cite[Chapter~6]{RockafellarWets}. Those cones therefore play a crucial role in constrained optimization to design algorithms and to formulate optimality conditions. As a matter of fact, they have recently appeared in various algorithms \cite{SchneiderUschmajew2015,ZhouEtAl2016,HosseiniUschmajew2019} and optimality conditions \cite{HosseiniLukeUschmajew2019,LiSongXiu2019,HaLiuFoygel2020} for both smooth and nonsmooth low-rank optimization, where the feasible set is
\begin{equation}
\label{eq:RealDeterminantalVariety}
\R_{\le r}^{m \times n} := \{X \in \R^{m \times n} \mid \rank X \le r\}
\end{equation}
for some positive integers $m$, $n$ and $r$ such that $r < \min\{m,n\}$.

In this paper, we mainly focus on the correspondence that maps each $X \in \R_{\le r}^{m \times n}$ to the tangent cone to $\R_{\le r}^{m \times n}$ at $X$. After preliminaries in Section~\ref{sec:Preliminaries}, we prove in Section~\ref{sec:Rank2x2BlockMatrices} fundamental linear algebra propositions that we use in Section~\ref{sec:MainResult} to prove our main result, Theorem~\ref{theorem:MainResult}, in which we compute inner and outer limits of this correspondence and draw conclusions concerning its continuity.
Such continuity results are required in order to try to strengthen the convergence analysis of the Riemannian rank-adaptive method proposed in~\cite{ZhouEtAl2016}.
Then, we deduce in Section~\ref{sec:ComplementaryResults} results about the continuity of other tangent and normal cones to $\R_{\le r}^{m \times n}$. Finally, we show in Section~\ref{sec:ConnectionRegularityWhitneyStratification} that Theorem~\ref{theorem:MainResult} includes as a particular case the $a$-regularity of the Whitney stratification of $\R_{\le r}^{m \times n}$ following from the fact that $\R_{\le r}^{m \times n}$ is a real algebraic variety, called the real determinantal variety \cite{Harris}.

\section{Preliminaries}
\label{sec:Preliminaries}
In this section, after introducing in Section~\ref{subsec:EuclideanVectorSpaceRealMatrices} basic notation concerning the Euclidean vector space of real matrices and some of its submanifolds, we recall in Section~\ref{subsec:InnerOuterLimitsContinuityCorrespondences} the concepts of relative inner and outer semicontinuity of correspondences, then we review in Section~\ref{subsec:TangentNormalCones} five sorts of tangent and normal cones, and finally we list in Section~\ref{subsec:TangentNormalConesRealDeterminantalVariety} the available formulas enabling to evaluate these tangent and normal cones to the real determinantal variety.

\subsection{The Euclidean vector space of real matrices and some of its submanifolds}
\label{subsec:EuclideanVectorSpaceRealMatrices}
In this paper, $m$ and $n$ are positive integers, $\R^{m \times n}$ is endowed with the Frobenius inner product $\ip{\cdot}{\cdot}$, $\norm{\cdot}$ is the Frobenius norm and, for every $X \in \R^{m \times n}$ and every real number $\rho > 0$, $B(X,\rho) := \{Y \in \R^{m \times n} \mid \norm{X-Y} < \rho\}$ and $B[X,\rho] := \{Y \in \R^{m \times n} \mid \norm{X-Y} \le \rho\}$ are respectively the open and closed balls of center $X$ and radius $\rho$ in $\R^{m \times n}$.
A nonempty subset $\mathcal{S}$ of $\R^{m \times n}$ is said to be locally closed at $X \in \mathcal{S}$ if $\mathcal{S} \cap B[X,\delta]$ is closed for some real number $\delta > 0$.
A nonempty subset $\mathcal{C}$ of $\R^{m \times n}$ is said to be a cone if, for every $X \in \mathcal{C}$ and every real number $\lambda \ge 0$, $\lambda X \in \mathcal{C}$.
For every nonempty subset $\mathcal{S}$ of $\R^{m \times n}$ and every $X \in \R^{m \times n}$, $d(X,\mathcal{S}) := \inf_{Y \in \mathcal{S}} \norm{X-Y}$ is the distance from $X$ to $\mathcal{S}$.
For every nonempty subset $\mathcal{S}$ of $\R^{m \times n}$, $\overline{\mathcal{S}} := \{X \in \R^{m \times n} \mid d(X,\mathcal{S}) = 0\}$ is a closed set called the closure of $\mathcal{S}$ and $\mathcal{S}^- := \{Y \in \R^{m \times n} \mid \ip{Y}{X} \le 0 \; \forall X \in \mathcal{S}\}$ is a closed convex cone called the (negative) polar of $\mathcal{S}$. If $\mathcal{S}$ is a linear subspace of $\R^{m \times n}$, then $\mathcal{S}^-$ is equal to the orthogonal complement $\mathcal{S}^\perp$ of $\mathcal{S}$. If $\emptyset \ne \mathcal{S}_1 \subseteq \mathcal{S}_2 \subseteq \R^{m \times n}$, then $\mathcal{S}_1^- \supseteq \mathcal{S}_2^-$. Finally, according to \cite[Corollary~6.21]{RockafellarWets}, for every cone $\mathcal{C} \subseteq \R^{m \times n}$, $\mathcal{C}^{--}$ is the closed convex hull of $\mathcal{C}$ and, in particular, the polar mapping is an involution on the set of all closed convex cones in $\R^{m \times n}$.

For every positive integer $r \le \min\{m,n\}$,
\begin{equation}
\label{eq:RealFixedRankManifold}
\R_r^{m \times n} := \{X \in \R^{m \times n} \mid \rank X = r\}
\end{equation}
is the smooth manifold of $m \times n$ rank-$r$ real matrices \cite[Proposition~4.1]{HelmkeShayman1995}, $\st(r,n) := \{U \in \R^{n \times r} \mid U^\tp U = I_r\}$ is the Stiefel manifold \cite[\S 3.3.2]{AbsilMahonySepulchre} and $\mathcal{O}_n := \st(n,n)$ is the orthogonal group.
The set $\{P \in \R_r^{n \times n} \mid P^2 = P = P^\tp\} = \{UU^\tp \mid U \in \st(r,n)\}$ of orthogonal projections in $\R_r^{n \times n}$ can be identified with the Grassmann manifold $\grass(r,n)$ \cite[(2.1)]{BendokatZimmermannAbsil2020} and is therefore compact.
If $X \in \R^{m \times n}$ and $X = U \Sigma V^\tp$ is a thin SVD, then the Moore--Penrose generalized inverse of $X$ is $X^\mpinv = V \Sigma^{-1} U^\tp$, and $X X^\mpinv = U U^\tp$ and $X^\mpinv X = V V^\tp$ are respectively orthogonal projections onto $\im X = \im U$ and $\im X^\tp = \im V$. Also, $(X^\tp)^\mpinv = (X^\mpinv)^\tp =: X^{\mpinv\tp}$ and the function $\R_r^{m \times n} \to \R_r^{n \times m} : X \mapsto X^\mpinv$ is continuous.
We close this section by proving the following basic proposition concerning the point-set topology of $\R_r^{m \times n}$ that will be frequently used in the paper, often implicitly; some results it contains are mentioned in \cite[\S 3.1.5]{AbsilMahonySepulchre} and \cite[(1.3) and Theorem~3.1]{SchneiderUschmajew2015}. Before that, we extend the definitions of $\R_{\le r}^{m \times n}$ in \eqref{eq:RealDeterminantalVariety} and of $\R_r^{m \times n}$ in \eqref{eq:RealFixedRankManifold} to every $r \in \N$ with, of course, $\R_{\le r}^{m \times n} = \R^{m \times n}$ if $r \ge \min\{m,n\}$, $\R_r^{m \times n} = \emptyset$ if $r > \min\{m,n\}$ and $\R_0^{m \times n} = \R_{\le 0}^{m \times n} = \{0_{m \times n}\}$. We also let $\R_*^{m \times n} := \R_{\min\{m,n\}}^{m \times n}$ denote the set of $m \times n$ real matrices that have full rank. Finally, for every $r \in \N$, we write $\R_{< r}^{m \times n} := \R_{\le r}^{m \times n} \setminus \R_r^{m \times n}$, $\R_{> r}^{m \times n} := \R^{m \times n} \setminus \R_{\le r}^{m \times n}$ and $\R_{\ge r}^{m \times n} := \R^{m \times n} \setminus \R_{< r}^{m \times n}$.

\begin{proposition}
\label{prop:PointSetTopologyRealFixedRankManifold}
For every positive integer $r \le \min\{m,n\}$, $\R_r^{m \times n}$ has the following properties:
\begin{enumerate}
\item it is open if $r = \min\{m,n\}$ while its interior is empty otherwise;
\item its closure is $\R_{\le r}^{m \times n}$;
\item it is dense and relatively open in $\R_{\le r}^{m \times n}$;
\item it is locally closed at each of its points.
\end{enumerate}
\end{proposition}

\begin{proof}
For every positive integer $k \le \min\{m,n\}$ and every $X \in \R_k^{m \times n}$ the singular values of which are $\sigma_1 \ge \dots \ge \sigma_k > 0$, by the Eckart--Young theorem \cite{EckartYoung1936} and because arbitrarily small singular values can be added to $X$ if $k < \min\{m,n\}$,
\begin{equation}
\label{eq:DistanceToRealDeterminantalVariety}
d(X,\R_r^{m \times n}) = \left\{ \begin{array}{ll}
\sqrt{\sum_{i=r+1}^k \sigma_i^2} & \text{if } k > r,\\
0 & \text{otherwise}.
\end{array} \right.
\end{equation}
It follows that $\R_{\le r}^{m \times n} \subseteq \overline{\R_r^{m \times n}}$ and that, for every $X \in \R^{m \times n}$, $d(X,\R_{\le r}^{m \times n}) = d(X,\R_r^{m \times n})$.
Thus, for every $X \in \R_r^{m \times n}$ and every $\varepsilon \in \big(0, d(X,\R_{< r}^{m \times n})\big)$, $B(X,\varepsilon) \subseteq \R_{\ge r}^{m \times n}$ with, if $r < \min\{m,n\}$, $B(X,\varepsilon) \cap \R_{> r}^{m \times n} \ne \emptyset$.
This establishes the first point of the proposition. This also implies that $\R_{\ge r}^{m \times n}$ is open and thus that $\R_{\le r}^{m \times n}$ is closed. As a result, $\overline{\R_r^{m \times n}} \subseteq \R_{\le r}^{m \times n}$, which establishes the second point of the proposition, and $\R_r^{m \times n} = \R_{\le r}^{m \times n} \cap \R_{\ge r}^{m \times n}$ is relatively open in $\R_{\le r}^{m \times n}$. The third point of the proposition follows since a set is obviously dense in its closure.
The fourth point follows from the fact that $\R_r^{m \times n} \cap B[X,\frac{1}{2}d(X,\R_{< r}^{m \times n})]$ is closed for every $X \in \R_r^{m \times n}$.
\end{proof}

\subsection{Inner and outer limits, continuity of correspondences}
\label{subsec:InnerOuterLimitsContinuityCorrespondences}
This section is mostly based on \cite[Chapters 4 and 5]{RockafellarWets}.
For every sequence $(S_i)_{i \in \N}$ of sets in a metric space $(X,d_X)$, the two sets
\begin{align*}
\inlim_{i \to \infty} S_i := \big\{x \in X \mid \lim_{i \to \infty} d_X(x,S_i) = 0\big\},&&
\outlim_{i \to \infty} S_i := \big\{x \in X \mid \liminf_{i \to \infty} d_X(x,S_i) = 0\big\}
\end{align*}
are closed and respectively called the \emph{inner} and \emph{outer limits} of $(S_i)_{i \in \N}$ \cite[Definition~4.1, Exercise~4.2(a) and Proposition~4.4]{RockafellarWets}. If $S_i \ne \emptyset$ for every $i \in \N$, then $\inlim_{i \to \infty} S_i$ and $\outlim_{i \to \infty} S_i$ are respectively the sets of all possible limits and of all possible cluster points of sequences $(x_i)_{i \in \N}$ such that $x_i \in S_i$ for every $i \in \N$.
It is always true that $\inlim_{i \to \infty} S_i \subseteq \outlim_{i \to \infty} S_i$; if the inclusion is an equality, then $(S_i)_{i \in \N}$ is said to \emph{converge in the sense of Painlev\'{e}} and $\setlim_{i \to \infty} S_i := \inlim_{i \to \infty} S_i = \outlim_{i \to \infty} S_i$ is called the \emph{limit} of $(S_i)_{i \in \N}$.

A \emph{correspondence}, or a \emph{set-valued mapping}, is a triple $F := (A, B, G)$ where $A$ and $B$ are sets respectively called the \emph{set of departure} and the \emph{set of destination} of $F$, and $G$ is a subset of $A \times B$ called the \emph{graph} of $F$.
If $F := (A, B, G)$ is a correspondence, written $F : A \setmapsto B$, then the \emph{image} of $x \in A$ by $F$ is $F(x) := \{y \in B \mid (x,y) \in G\}$ and the \emph{domain} of $F$ is $\dom F := \{x \in A \mid F(x) \ne \emptyset\}$.

We now review a notion of continuity for correspondences $F : X \setmapsto Y$ where $(X,d_X)$ and $(Y,d_Y)$ are two metric spaces. Let $S$ be a nonempty subset of $\dom F$ and $x$ be in the closure $\overline{S}$ of $S$. The two sets
\begin{align}
\label{eq:InnerLimitCorrespondence}
\inlim_{S \ni z \to x} F(z)
&:= \bigcap_{S \ni x_i \to x} \inlim_{i \to \infty} F(x_i)
= \big\{y \in Y \mid \lim_{S \ni z \to x} d_Y(y,F(z)) = 0\big\},\\
\label{eq:OuterLimitCorrespondence}
\outlim_{S \ni z \to x} F(z)
&:= \bigcup_{S \ni x_i \to x} \outlim_{i \to \infty} F(x_i)
= \big\{y \in Y \mid \liminf_{S \ni z \to x} d_Y(y,F(z)) = 0\big\}
\end{align}
are closed and respectively called the inner and outer limits of $F$ relative to $S$ at $x$ \cite[5(1)]{RockafellarWets}. Clearly, $\inlim_{S \ni z \to x} F(z) \subseteq \outlim_{S \ni z \to x} F(z)$; if the inclusion is an equality, then $\setlim_{S \ni z \to x} F(z) := \inlim_{S \ni z \to x} F(z) = \outlim_{S \ni z \to x} F(z)$ is called the limit of $F$ relative to $S$ at $x$.
According to \cite[Definition~5.4]{RockafellarWets}, $F$ is said to be \emph{inner semicontinuous} relative to $S$ at $x \in \overline{S} \cap \dom F$ if $\inlim_{S \ni z \to x} F(z) \supseteq F(x)$, \emph{outer semicontinuous} relative to $S$ at $x$ if $\outlim_{S \ni z \to x} F(z) \subseteq F(x)$ and \emph{continuous} relative to $S$ at $x$ if $F$ is both inner and outer semicontinuous relative to $S$ at $x$, i.e., $\setlim_{S \ni z \to x} F(z) = F(x)$.
If $S = \dom F$, then we omit the ``relative to $S$'' for brevity.
Let us mention two facts that will be frequently used in the paper.
First, if $S' \subseteq S$ and $x \in \overline{S'}$, then
\begin{equation}
\label{eq:MonotonicityInnerOuterLimits}
\inlim_{S \ni z \to x} F(z)
\subseteq \inlim_{S' \ni z \to x} F(z)
\subseteq \outlim_{S' \ni z \to x} F(z)
\subseteq \outlim_{S \ni z \to x} F(z)
\end{equation}
and, in particular, the inner (or outer) semicontinuity of $F$ relative to $S$ at $x$ implies the inner (or outer) semicontinuity of $F$ relative to $S'$ at $x$.
Secondly, if $x \in S$, then \eqref{eq:InnerLimitCorrespondence} and \eqref{eq:OuterLimitCorrespondence} clearly imply that
\begin{equation}
\label{eq:NonDeletedInnerOuterLimits}
\inlim_{S \ni z \to x} F(z) \subseteq \overline{F(x)} \subseteq \outlim_{S \ni z \to x} F(z).
\end{equation}
As a result, in that case, $F$ is inner semicontinuous relative to $S$ at $x \in S$ if and only if $\inlim_{S \ni z \to x} F(z) = \overline{F(x)}$ and outer semicontinuous relative to $S$ at $x$ if and only if $\outlim_{S \ni z \to x} F(z) = F(x)$.

\subsection{Tangent and normal cones}
\label{subsec:TangentNormalCones}
In this section, mostly based on \cite[Chapter~6]{RockafellarWets}, we review, for the convenience of the reader, especially because various terminologies and notations can be found in the literature, the two tangent cones and the three normal cones that are considered in this paper. For every nonempty subset $\mathcal{S}$ of $\R^{m \times n}$, the tangent and normal cones to $\mathcal{S}$ are correspondences with sets of departure and of destination both equal to $\R^{m \times n}$ and domain equal to $\mathcal{S}$. In the rest of this section, $X$ is a point of a subset $\mathcal{S}$ of $\R^{m \times n}$.

The set
\begin{align}
\label{eq:TangentConeOuterLimit}
\tancone{\mathcal{S}}{X}
:=&\; \outlim_{t \to 0^+} \frac{\mathcal{S}-X}{t}\\
\label{eq:TangentConeDistance}
=&\; \left\{V \in \R^{m \times n} \mid \liminf_{t \to 0^+} \frac{d(X+tV,\mathcal{S})}{t} = 0\right\}\\
\label{eq:TangentConeSequence}
=&\; \left\{V \in \R^{m \times n} \mid \exists \begin{array}{l} (t_i)_{i \in \N} \text{ in } (0,\infty) \text{ converging to } 0 \\ (V_i)_{i \in \N} \text{ in } \R^{m \times n} \text{ converging to } V \end{array} : X+t_iV_i \in \mathcal{S} \; \forall i \in \N\right\}
\end{align}
is a closed cone called the \emph{(Bouligand) tangent cone} \cite[Definition~6.1 and Proposition~6.2]{RockafellarWets} or the \emph{contingent cone} \cite[Definition~1.8(i)]{Mordukhovich1} to $\mathcal{S}$ at $X$.
The equality between \eqref{eq:TangentConeOuterLimit} and \eqref{eq:TangentConeSequence} readily follows from the first equality in \eqref{eq:OuterLimitCorrespondence} while the equality between \eqref{eq:TangentConeOuterLimit} and \eqref{eq:TangentConeDistance} follows from the second equality in \eqref{eq:OuterLimitCorrespondence} and the identity
\begin{equation}
\label{eq:TangentVectorDistance}
\frac{d(X+tV, \mathcal{S})}{t} = d\Big(V, \frac{\mathcal{S}-X}{t}\Big)
\end{equation}
holding for every real number $t > 0$ and every $V \in \R^{m \times n}$.
The closedness of $\tancone{\mathcal{S}}{X}$ follows from the fact that it is an outer limit. The fact that $\tancone{\mathcal{S}}{X}$ is a cone is clear from \eqref{eq:TangentConeDistance}. Finally, let us observe that, if $\mathcal{S} \subseteq \mathcal{S}' \subseteq \R^{m \times n}$, then $\tancone{\mathcal{S}}{X} \subseteq \tancone{\mathcal{S}'}{X}$.

The set
\begin{align}
\label{eq:RegularTangentConeInnerLimit}
\regtancone{\mathcal{S}}{X}
:=&\; \inlim_{\substack{t \to 0^+ \\ \mathcal{S} \ni Z \to X}} \frac{\mathcal{S}-Z}{t}\\
\label{eq:RegularTangentConeDistance}
=&\; \left\{V \in \R^{m \times n} \mid \lim_{\substack{t \to 0^+ \\ \mathcal{S} \ni Z \to X}} \frac{d(Z+tV,\mathcal{S})}{t} = 0\right\}\\
\label{eq:RegularTangentConeSequence}
=&\; \left\{V \in \R^{m \times n} \mid \begin{array}{rl}
\multirow{2}{*}{$\forall$ \hspace*{-5mm}} & (t_i)_{i \in \N} \text{ in } (0,\infty) \text{ converging to } 0 \\
& (X_i)_{i \in \N} \text{ in } \mathcal{S} \text{ converging to } X \\
& \hspace*{7mm} \exists \hspace*{1mm} (V_i)_{i \in \N} \text{ in } \R^{m \times n} \text{ converging to } V : X_i+t_iV_i \in \mathcal{S} \; \forall i \in \N
\end{array} \right\}
\end{align}
is a closed convex cone \cite[Theorem~6.26]{RockafellarWets} called the \emph{regular tangent cone} \cite[Definition~6.25]{RockafellarWets} or the \emph{Clarke tangent cone} \cite[Definition~1.8(iii)]{Mordukhovich1} to $\mathcal{S}$ at $X$.
The equality between \eqref{eq:RegularTangentConeInnerLimit} and \eqref{eq:RegularTangentConeSequence} readily follows from the first equality in \eqref{eq:InnerLimitCorrespondence} while the equality between \eqref{eq:RegularTangentConeInnerLimit} and \eqref{eq:RegularTangentConeDistance} follows from the second equality in \eqref{eq:InnerLimitCorrespondence} and \eqref{eq:TangentVectorDistance}.
The closedness of $\regtancone{\mathcal{S}}{X}$ follows from the fact that it is an inner limit. The fact that $\regtancone{\mathcal{S}}{X}$ is a cone is clear from \eqref{eq:RegularTangentConeDistance}; the convexity of $\regtancone{\mathcal{S}}{X}$ is then readily verified using \eqref{eq:RegularTangentConeSequence}.
Those two tangent cones are related by \cite[Theorem~6.26]{RockafellarWets}: if $\mathcal{S}$ is locally closed at $X$, then $\regtancone{\mathcal{S}}{X}$ is the inner limit of $\tancone{\mathcal{S}}{\cdot}$ at $X$, i.e., 
\begin{equation}
\label{eq:RegularTangentConeIsInnerLimitTangentCone}
\inlim_{Z \to X} \tancone{\mathcal{S}}{Z} = \regtancone{\mathcal{S}}{X}.
\end{equation}

The set
\begin{equation}
\label{eq:RegularNormalCone}
\regnorcone{\mathcal{S}}{X} := (\tancone{\mathcal{S}}{X})^-
\end{equation}
is called the \emph{regular normal cone} \cite[Definition~6.3 and Proposition~6.5]{RockafellarWets} or the \emph{prenormal cone} \cite[Definition~1.1(i) and Corollary~1.11]{Mordukhovich1} to $\mathcal{S}$ at $X$.

The set
\begin{equation}
\label{eq:NormalCone}
\norcone{\mathcal{S}}{X} := \outlim_{Z \to X} \regnorcone{\mathcal{S}}{Z}
\end{equation}
is a closed cone called the \emph{normal cone} to $\mathcal{S}$ at $X$ \cite[Definition~6.3 and Proposition~6.5]{RockafellarWets}. It is closed as an outer limit and it is a cone because of \cite[Exercise~4.14]{RockafellarWets}.
By \cite[Proposition~6.6]{RockafellarWets}, $\norcone{\mathcal{S}}{\cdot}$ is outer semicontinuous:
\begin{equation}
\label{eq:NormalConeOuterSemicontinuous}
\outlim_{Z \to X} \norcone{\mathcal{S}}{Z} = \norcone{\mathcal{S}}{X}.
\end{equation}
If $\mathcal{S}$ is locally closed at $X$, then $\regtancone{\mathcal{S}}{X} = (\norcone{\mathcal{S}}{X})^-$ \cite[Theorem~6.28(b)]{RockafellarWets} and, by \cite[Theorem~1.6]{Mordukhovich1}, $\norcone{\mathcal{S}}{X}$ corresponds to the \emph{(basic) normal cone} to $\mathcal{S}$ at $X$ defined in \cite[Definition~1.1(ii)]{Mordukhovich1} also called the \emph{Mordukhovich normal cone} in \cite{HosseiniLukeUschmajew2019}.

The set
\begin{equation}
\label{eq:ConvexifiedNormalCone}
\connorcone{\mathcal{S}}{X} := (\regtancone{\mathcal{S}}{X})^-
\end{equation}
is called the \emph{Clarke normal cone} to $\mathcal{S}$ at $X$ \cite[p. 17]{Mordukhovich1}.
Since the polar mapping is an involution on the set of all closed convex cones in $\R^{m \times n}$, $\regtancone{\mathcal{S}}{X} = (\connorcone{\mathcal{S}}{X})^-$.
If $\mathcal{S}$ is locally closed at $X$, then $(\norcone{\mathcal{S}}{X})^{--} = (\regtancone{\mathcal{S}}{X})^- = \connorcone{\mathcal{S}}{X}$, i.e., $\connorcone{\mathcal{S}}{X}$ is the closed convex hull of $\norcone{\mathcal{S}}{X}$. For this reason, $\connorcone{\mathcal{S}}{X}$ is sometimes called the \emph{convexified normal cone} to $\mathcal{S}$ at $X$; see \cite[6(19)-6(20) and Exercise~6.38(a)]{RockafellarWets}.

If $\mathcal{S}$ is locally closed at $X$, then those three normal cones are nested as follows:
\begin{equation}
\label{eq:InclusionsNormalCones}
\regnorcone{\mathcal{S}}{X} \subseteq \norcone{\mathcal{S}}{X} \subseteq \connorcone{\mathcal{S}}{X},
\end{equation}
where the first inclusion follows from \eqref{eq:NonDeletedInnerOuterLimits} and holds even if $\mathcal{S}$ is not locally closed at $X$.

The five tangent or normal cones that have been introduced in this section are represented in the diagram of Figure~\ref{fig:FiveTangentOrNormalCones} based on \cite[Figure~6.17]{RockafellarWets}.
They generalize the concepts of tangent and normal subspaces from differential geometry. More precisely, according to \cite[Example~6.8]{RockafellarWets}, if $\mathcal{S}$ is a smooth manifold in $\R^{m \times n}$ around $X$, then $\tancone{\mathcal{S}}{X} = \regtancone{\mathcal{S}}{X}$, $\norcone{\mathcal{S}}{X} = \regnorcone{\mathcal{S}}{X} = \connorcone{\mathcal{S}}{X}$, and $\tancone{\mathcal{S}}{X}$ and $\norcone{\mathcal{S}}{X}$ are linear subspaces of $\R^{m \times n}$ that are the orthogonal complements of each other.
\begin{figure}[h]
\begin{center}
\begin{tikzpicture}[scale=2]
\draw (0,1) node {$\tancone{\mathcal{S}}{\cdot}$};
\draw (2,1) node {$\regtancone{\mathcal{S}}{\cdot}$};
\draw (0,0) node {$\regnorcone{\mathcal{S}}{\cdot}$};
\draw (2,0) node {$\norcone{\mathcal{S}}{\cdot}$};
\draw (4,0) node {$\connorcone{\mathcal{S}}{\cdot}$};
\draw[->] (0.3,1) -- node [midway, above] {{\footnotesize inner limit}} (1.7,1);
\draw[->] (0.3,0) -- node [midway, above] {{\footnotesize outer limit}} (1.7,0);
\draw[->] (0,0.85) -- node [midway, right] {{\footnotesize polar}} (0,0.15);
\draw[->] (2,0.15) -- node [midway, right] {{\footnotesize polar}} (2,0.85);
\draw[->] (2.3,0) -- node [midway, above] {{\footnotesize closed convex hull}} (3.7,0);
\draw[<->] (2.15,0.85) -- node [midway, above right] {{\footnotesize polar}} (3.85,0.15);
\end{tikzpicture}
\end{center}
\caption{Five tangent or normal cones to a locally closed subset $\mathcal{S}$ of $\R^{m \times n}$.}
\label{fig:FiveTangentOrNormalCones}
\end{figure}
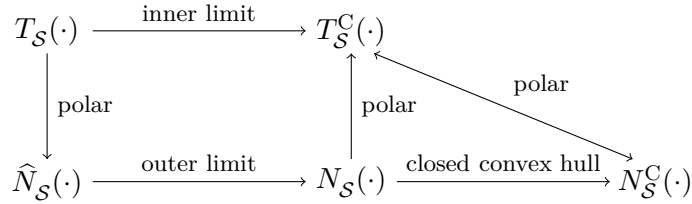

\subsection{Tangent and normal cones to the real determinantal variety}
\label{subsec:TangentNormalConesRealDeterminantalVariety}
In this section, we build on the material reviewed in Section~\ref{subsec:TangentNormalCones} and focus on the case where $\mathcal{S} = \R_{\le \oshort{r}}^{m \times n}$ for positive integers $\oshort{r} \le \min\{m,n\}$. More precisely, we gather in Theorem~\ref{theorem:TangentAndNormalConesRealDeterminantalVariety} the explicit formulas available to evaluate the correspondences $\tancone{\R_{\le \oshort{r}}^{m \times n}}{\cdot}$, $\regtancone{\R_{\le \oshort{r}}^{m \times n}}{\cdot}$, $\norcone{\R_{\le \oshort{r}}^{m \times n}}{\cdot}$, $\regnorcone{\R_{\le \oshort{r}}^{m \times n}}{\cdot}$ and $\connorcone{\R_{\le \oshort{r}}^{m \times n}}{\cdot}$, the semicontinuity of which we will investigate in Sections~\ref{sec:MainResult} and \ref{sec:ComplementaryResults}.

\begin{theorem}
\label{theorem:TangentAndNormalConesRealDeterminantalVariety}
Let $r$ and $\oshort{r}$ be positive integers such that $r \le \oshort{r} \le \min\{m,n\}$, and $X \in \R_r^{m \times n}$.
Then,
\begin{align}
\label{eq:ConvexifiedNormalConeRealDeterminantalVariety}
\connorcone{\R_{\le \oshort{r}}^{m \times n}}{X} &= \norcone{\R_r^{m \times n}}{X} = (\im X)^\perp \otimes (\im X^\tp)^\perp,\\
\label{eq:RegularTangentConeRealDeterminantalVariety}
\regtancone{\R_{\le \oshort{r}}^{m \times n}}{X} &= \tancone{\R_r^{m \times n}}{X} = (\norcone{\R_r^{m \times n}}{X})^\perp,\\
\label{eq:TangentConeRealDeterminantalVariety}
\tancone{\R_{\le \oshort{r}}^{m \times n}}{X} &= \tancone{\R_r^{m \times n}}{X} \oplus \norcone{\R_r^{m \times n}}{X} \cap \R_{\le \oshort{r}-r}^{m \times n},\\
\label{eq:NormalConeRealDeterminantalVariety}
\norcone{\R_{\le \oshort{r}}^{m \times n}}{X} &= \norcone{\R_r^{m \times n}}{X} \cap \R_{\le \min\{m,n\}-\oshort{r}}^{m \times n}.
\end{align}
If $r < \oshort{r}$, then
\begin{equation}
\label{eq:RegularNormalConeRealDeterminantalVariety}
\regnorcone{\R_{\le \oshort{r}}^{m \times n}}{X} = \{0_{m \times n}\}.
\end{equation}
More explicitly, if $U \in \R_*^{m \times r}$, $U_\perp \in \R_*^{m \times m-r}$, $V \in \R_*^{n \times r}$ and $V_\perp \in \R_*^{n \times n-r}$ are such that $\im U = \im X$, $\im U_\perp = (\im X)^\perp$, $\im V = \im X^\tp$ and $\im V_\perp = (\im X^\tp)^\perp$, then $\eta \in \tancone{\R_{\le \oshort{r}}^{m \times n}}{X}$ if and only if there exist $A \in \R^{r \times r}$, $B \in \R^{r \times n-r}$, $C \in \R^{m-r \times r}$ and $D \in \R_{\le \oshort{r}-r}^{m-r \times n-r}$ such that
\begin{equation}
\label{eq:TangentConeRealDeterminantalVarietyMatrix}
\eta = [U \; U_\perp] \begin{bmatrix} A & B \\ C & D \end{bmatrix} [V \; V_\perp]^\tp
\end{equation}
in which case $A = U^\mpinv \eta V^{\mpinv\tp}$, $B = U^\mpinv \eta V_\perp^{\mpinv\tp}$, $C = U_\perp^\mpinv \eta V^{\mpinv\tp}$ and $D = U_\perp^\mpinv \eta V_\perp^{\mpinv\tp}$.
\end{theorem}

\begin{proof}
The tangent space $\tancone{\R_r^{m \times n}}{X}$ is given in \cite[Proposition 4.1]{HelmkeShayman1995} and the normal space $\norcone{\R_r^{m \times n}}{X}$ is equal to $(\tancone{\R_r^{m \times n}}{X})^\perp$.
See \cite[Corollary~3.2]{HosseiniLukeUschmajew2019} for \eqref{eq:ConvexifiedNormalConeRealDeterminantalVariety}, \cite[Theorem~3.1]{LiSongXiu2019} for \eqref{eq:RegularTangentConeRealDeterminantalVariety}, \cite[Theorem~3.2]{SchneiderUschmajew2015} for \eqref{eq:TangentConeRealDeterminantalVariety}, \cite[Theorem~3.1]{HosseiniLukeUschmajew2019} for \eqref{eq:NormalConeRealDeterminantalVariety} and \cite[Corollary~2.3]{HosseiniLukeUschmajew2019} for \eqref{eq:RegularNormalConeRealDeterminantalVariety}.
Let us provide a shorter proof of \eqref{eq:RegularTangentConeRealDeterminantalVariety} than in \cite{LiSongXiu2019}: $\regtancone{\R_{\le \oshort{r}}^{m \times n}}{X} = (\connorcone{\R_{\le \oshort{r}}^{m \times n}}{X})^- = (\norcone{\R_r^{m \times n}}{X})^- = (\norcone{\R_r^{m \times n}}{X})^\perp = \tancone{\R_r^{m \times n}}{X}$. The last two equalities hold because, as $\R_r^{m \times n}$ is a smooth manifold, $\norcone{\R_r^{m \times n}}{X}$ is a linear space the orthogonal complement of which is $\tancone{\R_r^{m \times n}}{X}$.
\end{proof}

\section{On the rank of $2 \times 2$ block matrices}
\label{sec:Rank2x2BlockMatrices}
In this section, $A \in \R^{k \times k}$, $B \in \R^{k \times q}$, $C \in \R^{p \times k}$ and $D \in \R^{p \times q}$ for positive integers $k$, $p$ and $q$, $s \in \N$ with $s \le \min\{p,q\}$ and, for every $\ell \in \N$, $J_\ell := [\delta_{i,\ell+1-j}]_{i,j=1}^\ell = \left[\begin{smallmatrix} & & 1 \\ & \reflectbox{$\ddots$} & \\ 1 & & \end{smallmatrix}\right]$ is the reversal matrix.

The following is a basic linear algebra result, but we could not find it in the literature, and it will play an instrumental role in the proof of Proposition~\ref{prop:UpperBoundOuterLimitTangentConeRealDeterminantalVariety}.

\begin{proposition}
\label{prop:BoundRank2*2BlockMatrices}
If $\rank D \le s$, then the following upper bound holds and is tight:
\begin{equation*}
\rank \left[\begin{smallmatrix} A & B \\ C & D \end{smallmatrix}\right] \le k + \min\{k+s, p, q\}.
\end{equation*}
\end{proposition}

\begin{proof}
If $k+s \ge \min\{p,q\}$, then the bound holds because $\left[\begin{smallmatrix} A & B \\ C & D \end{smallmatrix}\right] \in \R^{k+p \times k+q}$.
If $k+s < \min\{p,q\}$, then the bound holds because
$
\rank \left[\begin{smallmatrix} A & B \\ C & D \end{smallmatrix}\right]
\le \rank \left[\begin{smallmatrix} A & B \\ C & 0_{p \times q} \end{smallmatrix}\right] + \rank \left[\begin{smallmatrix} 0_k & \\ & D \end{smallmatrix}\right]
\le 2k + s
$.
In both cases, the bound is reached for $\left[\begin{smallmatrix} A & B \\ C & D \end{smallmatrix}\right] := \diag(J_{k+\min\{k, \min\{p,q\}-s\}}, J_s, 0_{p-\min\{k+s, p, q\} \times q-\min\{k+s, p, q\}})$.
\end{proof}

The next result is a sort of converse of the preceding one and will be invoked in the proof of Proposition~\ref{prop:MaximalOuterLimitTangentConeRealDeterminantalVariety}.

\begin{proposition}
\label{prop:PermutationRank2*2BlockMatrices}
If
$
\rank \left[\begin{smallmatrix} A & B \\ C & D \end{smallmatrix}\right] \le 2k+s
$,
then there exist $U \in \mathcal{O}_{k+p}$ and $V \in \mathcal{O}_{k+q}$ such that
$
\left[\begin{smallmatrix} A & B \\ C & D \end{smallmatrix}\right] = U \left[\begin{smallmatrix} A' & B' \\ C' & D' \end{smallmatrix}\right] V^\tp
$
with $A' \in \R^{k \times k}$, $B' \in \R^{k \times q}$, $C' \in \R^{p \times k}$, $D' \in \R^{p \times q}$ and
\begin{equation*}
\rank D' \le \min\{s, \max\{\min\{p,q\}-k, 0\}\}.
\end{equation*}
\end{proposition}

\begin{proof}
Let
$
\left[\begin{smallmatrix} A & B \\ C & D \end{smallmatrix}\right] = [\tilde{U} \; U_\perp] \left[\begin{smallmatrix} \Sigma & \\ & 0_{k+p-r \times k+q-r} \end{smallmatrix}\right] [\tilde{V} \; V_\perp]^\tp
$
be an SVD, where $r := \rank \left[\begin{smallmatrix} A & B \\ C & D \end{smallmatrix}\right] \le k + \min\{k+s, p, q\}$.
Then, it suffices to take
$
\left[\begin{smallmatrix} A' & B' \\ C' & D' \end{smallmatrix}\right] := \left[\begin{smallmatrix} J_r \Sigma & \\ & 0_{k+p-r \times k+q-r} \end{smallmatrix}\right]
$,
$U := [\tilde{U} J_r \; U_\perp]$ and $V := [\tilde{V} \; V_\perp]$.
\end{proof}

\section{Main result}
\label{sec:MainResult}
In this section, for positive integers $\ushort{r}$, $r$ and $\oshort{r}$ such that $\ushort{r} \le r \le \oshort{r} < \min\{m,n\}$, we compute inner and outer limits relative to $\R_r^{m \times n}$ and $\R_{\le r}^{m \times n}$ at $X \in \R_{\ushort{r}}^{m \times n}$ of the correspondence
\begin{align}
\label{eq:CorrespondenceTangentConeRealDeterminantalVariety}
\tancone{\R_{\le \oshort{r}}^{m \times n}}{\cdot} : \R^{m \times n} \setmapsto \R^{m \times n} : X \mapsto \tancone{\R_{\le \oshort{r}}^{m \times n}}{X},&&
\dom \tancone{\R_{\le \oshort{r}}^{m \times n}}{\cdot} = \R_{\le \oshort{r}}^{m \times n},
\end{align}
and draw conclusions on its relative (semi)continuity.
More precisely, we prove the following result.

\begin{theorem}
\label{theorem:MainResult}
For every sequence $(X_i)_{i \in \N}$ in $\R_{\le r}^{m \times n}$ converging to $X \in \R_{\ushort{r}}^{m \times n}$,
\begin{equation}
\label{eq:LowerUpperBoundsInnerOuterLimitsTangentConeRealDeterminantalVariety}
\tancone{\R_{\le \oshort{r}-r+\ushort{r}}^{m \times n}}{X}
\subseteq \inlim_{i \to \infty} \tancone{\R_{\le \oshort{r}}^{m \times n}}{X_i}
\subseteq \outlim_{i \to \infty} \tancone{\R_{\le \oshort{r}}^{m \times n}}{X_i}
\subseteq \tancone{\R_{\le \oshort{r}+r-\ushort{r}}^{m \times n}}{X}.
\end{equation}
Moreover, if $\ushort{r} < r$, then, for every $X \in \R_{\ushort{r}}^{m \times n}$, there exists a sequence $(X_i)_{i \in \N}$ in $\R_r^{m \times n}$ converging to $X$ such that
\begin{equation}
\label{eq:MinimalMaximalInnerOuterLimitsTangentConeRealDeterminantalVariety}
\tancone{\R_{\le \oshort{r}-r+\ushort{r}}^{m \times n}}{X}
= \inlim_{i \to \infty} \tancone{\R_{\le \oshort{r}}^{m \times n}}{X_i}
\subsetneq \outlim_{i \to \infty} \tancone{\R_{\le \oshort{r}}^{m \times n}}{X_i}
= \tancone{\R_{\le \oshort{r}+r-\ushort{r}}^{m \times n}}{X}.
\end{equation}
Thus, for every $X \in \R_{\ushort{r}}^{m \times n}$,
\begin{align}
\label{eq:RelativeInnerLimitTangentConeRealDeterminantalVariety}
\inlim_{\R_{\le r}^{m \times n} \ni Z \to X} \tancone{\R_{\le \oshort{r}}^{m \times n}}{Z}
&= \inlim_{\R_r^{m \times n} \ni Z \to X} \tancone{\R_{\le \oshort{r}}^{m \times n}}{Z}
= \tancone{\R_{\le \oshort{r}-r+\ushort{r}}^{m \times n}}{X},\\
\label{eq:RelativeOuterLimitTangentConeRealDeterminantalVariety}
\outlim_{\R_{\le r}^{m \times n} \ni Z \to X} \tancone{\R_{\le \oshort{r}}^{m \times n}}{Z}
&= \outlim_{\R_r^{m \times n} \ni Z \to X} \tancone{\R_{\le \oshort{r}}^{m \times n}}{Z}
= \tancone{\R_{\le \oshort{r}+r-\ushort{r}}^{m \times n}}{X}.
\end{align}
In particular, the correspondence $\tancone{\R_{\le \oshort{r}}^{m \times n}}{\cdot}$ is:
\begin{itemize}
\item continuous relative to $\R_{\le r}^{m \times n}$ at every $X \in \R_r^{m \times n}$;
\item neither inner nor outer semicontinuous relative to $\R_r^{m \times n}$ at every $X \in \R_{< r}^{m \times n}$.
\end{itemize}
\end{theorem}

\subsection{Proof of Theorem~\ref{theorem:MainResult}}
By \eqref{eq:InnerLimitCorrespondence}, \eqref{eq:OuterLimitCorrespondence} and \eqref{eq:MonotonicityInnerOuterLimits}, \eqref{eq:RelativeInnerLimitTangentConeRealDeterminantalVariety} and \eqref{eq:RelativeOuterLimitTangentConeRealDeterminantalVariety} readily follow from \eqref{eq:LowerUpperBoundsInnerOuterLimitsTangentConeRealDeterminantalVariety} and \eqref{eq:MinimalMaximalInnerOuterLimitsTangentConeRealDeterminantalVariety}. Propositions~\ref{prop:LowerBoundInnerLimitTangentConeRealDeterminantalVariety}, \ref{prop:MinimalInnerLimitTangentConeRealDeterminantalVariety}, \ref{prop:UpperBoundOuterLimitTangentConeRealDeterminantalVariety} and \ref{prop:MaximalOuterLimitTangentConeRealDeterminantalVariety} respectively state the first inclusion of \eqref{eq:LowerUpperBoundsInnerOuterLimitsTangentConeRealDeterminantalVariety}, the first equality of \eqref{eq:MinimalMaximalInnerOuterLimitsTangentConeRealDeterminantalVariety}, the last inclusion of \eqref{eq:LowerUpperBoundsInnerOuterLimitsTangentConeRealDeterminantalVariety} and the last equality of \eqref{eq:MinimalMaximalInnerOuterLimitsTangentConeRealDeterminantalVariety}. Proposition~\ref{prop:RelativeContinuityTangentConeRealDeterminantalVariety} specifically focuses on the case where $\ushort{r} = r$.

In a nutshell, those five propositions concern inner and outer limits of $\big(\tancone{\R_{\le \oshort{r}}^{m \times n}}{X_i}\big)_{i \in \N}$ for sequences $(X_i)_{i \in \N}$ in $\R_{\le r}^{m \times n}$ converging to $X \in \R_{\ushort{r}}^{m \times n}$. More precisely, Propositions~\ref{prop:LowerBoundInnerLimitTangentConeRealDeterminantalVariety} and \ref{prop:UpperBoundOuterLimitTangentConeRealDeterminantalVariety} respectively provide a lower bound on the inner limits and an upper bound on the outer limits while Propositions~\ref{prop:MinimalInnerLimitTangentConeRealDeterminantalVariety} and \ref{prop:MaximalOuterLimitTangentConeRealDeterminantalVariety} show that these two bounds can be reached for a particular sequence $(X_i)_{i \in \N}$. This is summarized in Table~\ref{tab:SchemeProofMainResult}.
\begin{table}[h]
\begin{center}
\begin{spacing}{1.2}
\begin{tabular}{l|c|c|}
\cline{2-3}
& Inner limit & Outer limit\\
\hline
\multicolumn{1}{|l|}{Bound \eqref{eq:LowerUpperBoundsInnerOuterLimitsTangentConeRealDeterminantalVariety}} & Proposition~\ref{prop:LowerBoundInnerLimitTangentConeRealDeterminantalVariety} & Proposition~\ref{prop:UpperBoundOuterLimitTangentConeRealDeterminantalVariety}\\
\hline
\multicolumn{1}{|l|}{Reachability of the bound \eqref{eq:MinimalMaximalInnerOuterLimitsTangentConeRealDeterminantalVariety}} & Proposition~\ref{prop:MinimalInnerLimitTangentConeRealDeterminantalVariety} & Proposition~\ref{prop:MaximalOuterLimitTangentConeRealDeterminantalVariety}\\
\hline
\end{tabular}
\end{spacing}
\end{center}
\caption{Inner and outer limits of $\big(\tancone{\R_{\le \oshort{r}}^{m \times n}}{X_i}\big)_{i \in \N}$ for sequences $(X_i)_{i \in \N}$ in $\R_{\le r}^{m \times n}$ converging to $X \in \R_{\ushort{r}}^{m \times n}$ with $\ushort{r} < r$.}
\label{tab:SchemeProofMainResult}
\end{table}

In view of the explicit formula \eqref{eq:TangentConeRealDeterminantalVarietyMatrix} for $\tancone{\R_{\le \oshort{r}}^{m \times n}}{\cdot}$, proving those five propositions requires to study the convergence of matrix representatives of $\im X_i$ and $\im X_i^\tp$ and of their respective orthogonal complements. We focus on that task in three lemmas on which our five propositions are based. Lemma~\ref{lemma:ConvergenceTangentConeDecompositionConstantRank} deals with the simple case where $\ushort{r} = r$ and serves as a basis for Propositions~\ref{prop:LowerBoundInnerLimitTangentConeRealDeterminantalVariety} and \ref{prop:RelativeContinuityTangentConeRealDeterminantalVariety} while Lemmas~\ref{lemma:ConvergenceTangentConeDecompositionDecreaseRank} and \ref{lemma:DenseTangentConeDecompositionDecreaseRank} both consider the case where $\ushort{r} < r$. Lemma~\ref{lemma:ConvergenceTangentConeDecompositionDecreaseRank} plays a prominent role in the proof of Proposition~\ref{prop:UpperBoundOuterLimitTangentConeRealDeterminantalVariety} whereas Lemma~\ref{lemma:DenseTangentConeDecompositionDecreaseRank} provides the particular sequence mentioned in \eqref{eq:MinimalMaximalInnerOuterLimitsTangentConeRealDeterminantalVariety}.

\begin{lemma}
\label{lemma:ConvergenceTangentConeDecompositionConstantRank}
Let $(X_i)_{i \in \N}$ be a sequence in $\R_r^{m \times n}$ converging to $X \in \R_r^{m \times n}$ and $$X = [U \; U_\perp] \begin{bmatrix} \Sigma & \\ & 0_{m-r \times n-r} \end{bmatrix} [V \; V_\perp]^\tp$$ be an SVD.
Then, there exist sequences $(U_i)_{i \in \N}$ in $\R_*^{m \times r}$, $(U_{i\perp})_{i \in \N}$ in $\R_*^{m \times m-r}$, $(V_i)_{i \in \N}$ in $\R_*^{n \times r}$ and $(V_{i\perp})_{i \in \N}$ in $\R_*^{n \times n-r}$ respectively converging to $U$, $U_\perp$, $V$ and $V_\perp$, and such that, for every $i \in \N$, $\im U_i = \im X_i$, $\im U_{i\perp} = (\im X_i)^\perp$, $\im V_i = \im X_i^\tp$ and $\im V_{i\perp} = (\im X_i^\tp)^\perp$.
\end{lemma}

\begin{proof}
We define the required sequences using orthogonal projections. For example, for every $i \in \N$ large enough, $U_i$ will be defined as the orthogonal projection of $U$ onto $\im X_i$, that is $(X_i X_i^\mpinv) U$. Indeed, by continuity, every sequence defined that way converges to $(X X^\mpinv) U = U$. Furthermore, $\im (X_i X_i^\mpinv U) \subseteq \im X_i$ for every $i \in \N$ and this inclusion actually becomes an equality when $i$ is large enough since $\rank (X_i X_i^\mpinv U)$ then becomes equal to $r$. Indeed, since $U \in \st(r,m)$, all its singular values are equal to $1$ and therefore $d(U,\R_{< r}^{m \times r}) = 1$, which implies in particular that $B(U,1) \subseteq \R_*^{m \times r}$. However, since $(X_i X_i^\mpinv U)_{i\in\N}$ converges to $U$, there exists $i_* \in \N$ such that $X_i X_i^\mpinv U \in B(U,1)$ for every integer $i > i_*$. In conclusion, we define $U_i := X_i X_i^\mpinv U$ for every integer $i > i_*$ and choose $U_i \in \R_*^{m \times r}$ such that $\im U_i = \im X_i$ for every $i \in \{0,\dots, i_*\}$.
The same process can be used to define the other required sequences: for every $i \in \N$ large enough, we define $U_{i\perp} := (I_m - X_i X_i^\mpinv) U_\perp$, $V_i := (X_i^\mpinv X_i) V$ and $V_{i\perp} := (I_n - X_i^\mpinv X_i) V_\perp$, and complete the definition for the other indices in order to meet the desired conditions.
\end{proof}

\begin{lemma}
\label{lemma:ConvergenceTangentConeDecompositionDecreaseRank}
Let $\ushort{r} < r$, $(X_i)_{i \in \N}$ be a sequence in $\R_r^{m \times n}$ converging to $X \in \R_{\ushort{r}}^{m \times n}$ and $X = U \Sigma V^\tp$ be a thin SVD.
Then, there exist
$\bar{U}_\perp \in \st(r-\ushort{r},m)$, $U_\perp \in \st(m-r,m)$, $\bar{V}_\perp \in \st(r-\ushort{r},n)$, $V_\perp \in \st(n-r,n)$,
a strictly increasing sequence $(i_k)_{k \in \N}$ in $\N$,
and sequences $(U_i)_{i \in \N}$ in $\R_*^{m \times \ushort{r}}$, $(\bar{U}_i)_{i \in \N}$ in $\R_*^{m \times r-\ushort{r}}$, $(U_{i\perp})_{i \in \N}$ in $\R_*^{m \times m-r}$, $(V_i)_{i \in \N}$ in $\R_*^{n \times \ushort{r}}$, $(\bar{V}_i)_{i \in \N}$ in $\R_*^{n \times r-\ushort{r}}$ and $(V_{i\perp})_{i \in \N}$ in $\R_*^{n \times n-r}$
satisfying the following properties:
\begin{enumerate}
\item $\im U = \im X$, $\im [\bar{U}_\perp \; U_\perp] = (\im X)^\perp$, $\im V = \im X^\tp$ and $\im [\bar{V}_\perp \; V_\perp] = (\im X^\tp)^\perp$;
\item for all $i \in \N$, $\im [U_i \; \bar{U}_i] = \im X_i$, $\im U_{i\perp} = (\im X_i)^\perp$, $\im [V_i \; \bar{V}_i] = \im X_i^\tp$ and $\im V_{i\perp} = (\im X_i^\tp)^\perp$;
\item $\lim\limits_{i \to \infty} U_i = U$, $\lim\limits_{k \to \infty} \bar{U}_{i_k} = \bar{U}_\perp$, $\lim\limits_{k \to \infty} U_{i_k\perp} = U_\perp$, $\lim\limits_{i \to \infty} V_i = V$, $\lim\limits_{k \to \infty} \bar{V}_{i_k} = \bar{V}_\perp$, $\lim\limits_{k \to \infty} V_{i_k\perp} = V_\perp$.
\end{enumerate}
\end{lemma}

\begin{proof}
For every $i \in \N$, let ${\displaystyle \ushort{X}_i \in \argmin\limits_{\ushort{X} \in \R_{\ushort{r}}^{m \times n}} \norm{X_i-\ushort{X}}}$ and $\tilde{X}_i := X_i - \ushort{X}_i \in \R_{r-\ushort{r}}^{m \times n}$.
Observe that:
\begin{itemize}
\item $\ushort{X}_i \to X$ and $\ushort{X}_i^\mpinv \to X^\mpinv$ as $i \to \infty$;
\item $\tilde{X}_i \to 0_{m \times n}$ as $i \to \infty$;
\item $\ushort{X}_i^\tp \tilde{X}_i = 0_n$ and $\im \ushort{X}_i \oplus \im \tilde{X}_i = \im X_i$ for every $i \in \N$.
\end{itemize}
Since $(\ushort{X}_i \ushort{X}_i^\mpinv U)_{i \in \N}$ and $(\ushort{X}_i^\mpinv \ushort{X}_i V)_{i \in \N}$ respectively converge to $U \in \st(\ushort{r},m)$ and $V \in \st(\ushort{r},n)$, there exists $i_* \in \N$ such that $\ushort{X}_i \ushort{X}_i^\mpinv U \in B(U,1)$ and $\ushort{X}_i^\mpinv \ushort{X}_i V \in B(V,1)$ for every integer $i > i_*$.
For every $i \in \{0,\dots, i_*\}$, we choose $U_i \in \R_*^{m \times \ushort{r}}$ such that $\im U_i = \im \ushort{X}_i$ and $V_i \in \R_*^{n \times \ushort{r}}$ such that $\im V_i = \im \ushort{X}_i^\tp$.
For every integer $i > i_*$, we define $U_i := \ushort{X}_i \ushort{X}_i^\mpinv U \in \R_*^{m \times \ushort{r}}$ and $V_i := \ushort{X}_i^\mpinv \ushort{X}_i V \in \R_*^{n \times \ushort{r}}$.

As $(\tilde{X}_i \tilde{X}_i^\mpinv)_{i \in \N}$ is a sequence of orthogonal projections in $\R_{r-\ushort{r}}^{m \times m}$, it contains a subsequence $(\tilde{X}_{j_k} \tilde{X}_{j_k}^\mpinv)_{k \in \N}$ converging to an orthogonal projection that can be written as $\bar{U}_\perp \bar{U}_\perp^\tp$ with $\bar{U}_\perp \in \st(r-\ushort{r},m)$.
It holds that $\im \bar{U}_\perp \subseteq (\im X)^\perp$ since $\bar{U}_\perp^\tp X = 0_{r-\ushort{r} \times n}$ as $\tilde{X}_i \tilde{X}_i^\mpinv \ushort{X}_i = 0_{m \times n}$ for every $i \in \N$.
In the same way, $(\tilde{X}_{j_k}^\mpinv \tilde{X}_{j_k})_{k \in \N}$ is a sequence of orthogonal projections in $\R_{r-\ushort{r}}^{n \times n}$ and therefore contains a subsequence $(\tilde{X}_{i_k}^\mpinv \tilde{X}_{i_k})_{k \in \N}$ converging to an orthogonal projection that can be written as $\bar{V}_\perp \bar{V}_\perp^\tp$ with $\bar{V}_\perp \in \st(r-\ushort{r},n)$ and $\im \bar{V}_\perp \subseteq (\im X^\tp)^\perp$.

Since $(\tilde{X}_{i_k} \tilde{X}_{i_k}^\mpinv \bar{U}_\perp)_{k \in \N}$ and $(\tilde{X}_{i_k}^\mpinv \tilde{X}_{i_k} \bar{V}_\perp)_{k \in \N}$ respectively converge to $\bar{U}_\perp \in \st(r-\ushort{r},m)$ and $\bar{V}_\perp \in \st(r-\ushort{r},n)$, there exists $k_* \in \N$ such that $\tilde{X}_{i_k} \tilde{X}_{i_k}^\mpinv \bar{U}_\perp \in B(\bar{U}_\perp,1)$ and $\tilde{X}_{i_k}^\mpinv \tilde{X}_{i_k} \bar{V}_\perp \in B(\bar{V}_\perp,1)$ for every integer $k > k_*$.
For every integer $k > k_*$, we define $\bar{U}_{i_k} := \tilde{X}_{i_k} \tilde{X}_{i_k}^\mpinv \bar{U}_\perp \in \R_*^{m \times r-\ushort{r}}$ and $\bar{V}_{i_k} := \tilde{X}_{i_k}^\mpinv \tilde{X}_{i_k} \bar{V}_\perp \in \R_*^{n \times r-\ushort{r}}$.
We complete those definitions to obtain sequences $(\bar{U}_i)_{i \in \N}$ and $(\bar{V}_i)_{i \in \N}$ satisfying the required properties: for every $i \in \N \setminus \{i_k \mid k \in \N,\, k > k_*\}$, we choose $\bar{U}_i \in \R_*^{m \times r-\ushort{r}}$ such that $\im \bar{U}_i = \im \tilde{X}_i$ and $\bar{V}_i \in \R_*^{n \times r-\ushort{r}}$ such that $\im \bar{V}_i = \im \tilde{X}_i^\tp$.

Let $U_\perp \in \st(m-r,m)$ and $V_\perp \in \st(n-r,n)$ be such that $\im U_\perp = (\im [U \; \bar{U}_\perp])^\perp$ and $\im V_\perp = (\im [V \; \bar{V}_\perp])^\perp$.
For every $k \in \N$ large enough, we define $U_{i_k\perp} := (I_m - \ushort{X}_{i_k} \ushort{X}_{i_k}^\mpinv - \tilde{X}_{i_k} \tilde{X}_{i_k}^\mpinv) U_\perp \in \R_*^{m \times m-r}$ and $V_{i_k\perp} := (I_n - \ushort{X}_{i_k}^\mpinv \ushort{X}_{i_k} - \tilde{X}_{i_k}^\mpinv \tilde{X}_{i_k}) V_\perp \in \R_*^{n \times n-r}$, and complete these definitions to obtain sequences $(U_{i\perp})_{i \in \N}$ and $(V_{i\perp})_{i \in \N}$ satisfying the required properties.
\end{proof}

\begin{lemma}
\label{lemma:DenseTangentConeDecompositionDecreaseRank}
For every $X \in \R_{\ushort{r}}^{m \times n}$, if $\ushort{r} < r$, $X = U \diag(\sigma_1, \dots, \sigma_{\ushort{r}}) V^\tp$ is a thin SVD,
$
\mathcal{U}_\perp := \{\tilde{U} \in \st(m-\ushort{r},m) \mid \im \tilde{U} = (\im U)^\perp\}
$
and 
$\mathcal{V}_\perp := \{\tilde{V} \in \st(n-\ushort{r},n) \mid \im \tilde{V} = (\im V)^\perp\}
$,
then there exist sequences $(X_i)_{i \in \N}$ in $\R_r^{m \times n}$ and $(([\bar{U}_i \; U_{i\perp}],\, [\bar{V}_i \; V_{i\perp}]))_{i \in \N}$ in $\mathcal{U}_\perp \times \mathcal{V}_\perp$ such that:
\begin{enumerate}
\item $(X_i)_{i \in \N}$ converges to $X$;
\item for all $i \in \N$, $\im [U \; \bar{U}_i] = \im X_i$, $\im U_{i\perp} = (\im X_i)^\perp$, $\im [V \; \bar{V}_i] = \im X_i^\tp$ and $\im V_{i\perp} = (\im X_i^\tp)^\perp$;
\item the set of cluster points of $(([\bar{U}_i \; U_{i\perp}],\, [\bar{V}_i \; V_{i\perp}]))_{i \in \N}$ is $\mathcal{U}_\perp \times \mathcal{V}_\perp$.
\end{enumerate}
\end{lemma}

\begin{proof}
In view of \cite[Definition~1.2.17 and Proposition~1.2.18]{Willem}, the set $\mathcal{U}_\perp \times \mathcal{V}_\perp$ is separable and therefore contains a sequence $(([\bar{U}_i \; U_{i\perp}],\, [\bar{V}_i \; V_{i\perp}]))_{i \in \N}$ the set of cluster points of which is exactly $\mathcal{U}_\perp \times \mathcal{V}_\perp$.
Then, defining $X_i := X + \frac{\sigma_{\ushort{r}}}{i+1} \bar{U}_i \bar{V}_i^\tp$ for every $i \in \N$ yields a sequence $(X_i)_{i \in \N}$ satisfying the required properties.
\end{proof}

\begin{proposition}
\label{prop:LowerBoundInnerLimitTangentConeRealDeterminantalVariety}
For every sequence $(X_i)_{i \in \N}$ in $\R_{\le r}^{m \times n}$ converging to $X \in \R_{\ushort{r}}^{m \times n}$,
\begin{equation*}
\inlim_{i \to \infty} \tancone{\R_{\le \oshort{r}}^{m \times n}}{X_i} \supseteq \tancone{\R_{\le \oshort{r}-r+\ushort{r}}^{m \times n}}{X}.
\end{equation*}
\end{proposition}

\begin{proof}
Let $(X_i)_{i \in \N}$ be a sequence in $\R_{\le r}^{m \times n}$ converging to $X \in \R_{\ushort{r}}^{m \times n}$. For a given $\eta \in \tancone{\R_{\le \oshort{r}-r+\ushort{r}}^{m \times n}}{X}$, let us construct a sequence $(\eta_i)_{i\in\N}$ converging to $\eta$ and such that $\eta_i \in \tancone{\R_{\le \oshort{r}}^{m \times n}}{X_i}$ for every $i \in \N$, and the proof will be complete.
By \eqref{eq:DistanceToRealDeterminantalVariety}, there exists $i_* \in \N$ such that, for every integer $i > i_*$, $\rank X_i \ge \ushort{r}$. For every $i \in \{0, \dots, i_*\}$, let us choose $\eta_i \in \tancone{\R_{\le \oshort{r}}^{m \times n}}{X_i}$. Let us now complete the definition of $(\eta_i)_{i\in\N}$.
Let $X = [U \; U_\perp] \diag(\Sigma,0_{m-\ushort{r} \times n-\ushort{r}}) [V \; V_\perp]^\tp$ be an SVD. By \eqref{eq:TangentConeRealDeterminantalVarietyMatrix}, there exist $A \in \R^{\ushort{r} \times \ushort{r}}$, $B \in \R^{\ushort{r} \times n-\ushort{r}}$, $C \in \R^{m-\ushort{r} \times \ushort{r}}$ and $D \in \R_{\le \oshort{r}-r}^{m-\ushort{r} \times n-\ushort{r}}$ such that $\eta = [U \; U_\perp] \left[\begin{smallmatrix} A & B \\ C & D \end{smallmatrix}\right] [V \; V_\perp]^\tp$.
For every integer $i > i_*$, let $\ushort{X}_i \in \argmin_{\ushort{X} \in \R_{\ushort{r}}^{m \times n}} \norm{X_i-\ushort{X}}$ as in the proof of Lemma~\ref{lemma:ConvergenceTangentConeDecompositionDecreaseRank}. Let us apply Lemma~\ref{lemma:ConvergenceTangentConeDecompositionConstantRank} to $(\ushort{X}_i)_{i \in \N,\, i > i_*}$, the $r$ in Lemma~\ref{lemma:ConvergenceTangentConeDecompositionConstantRank} being $\ushort{r}$ here, and define, for every integer $i > i_*$, $\eta_i := [U_i \; U_{i\perp}] \left[\begin{smallmatrix} A & B \\ C & D \end{smallmatrix}\right] [V_i \; V_{i\perp}]^\tp$. Then, $(\eta_i)_{i \in \N}$ converges to $\eta$ and it remains to prove that $\eta_i \in \tancone{\R_{\le \oshort{r}}^{m \times n}}{X_i}$ for every integer $i > i_*$. Let $i \in \N$, $i > i_*$ and $r_i := \rank X_i$. The case where $r_i = \ushort{r}$ is trivial since it implies $\ushort{X}_i = X_i$. Let us therefore consider the case where $r_i > \ushort{r}$. Let $P_i \in \mathcal{O}_{m-\ushort{r}}$ be such that the first $r_i-\ushort{r}$ columns of $U_{i\perp}' := U_{i\perp} P_i$ together with $U_i$ span $\im X_i$. Likewise, let $Q_i \in \mathcal{O}_{n-\ushort{r}}$ be such that the first $r_i-\ushort{r}$ columns of $V_{i\perp}' := V_{i\perp} Q_i$ together with $V_i$ span $\im X_i^\tp$. Then, $\eta_i = [U_i \; U_{i\perp}'] \left[\begin{smallmatrix} A & B Q_i \\ P_i^\tp C & P_i^\tp D Q_i \end{smallmatrix}\right] [V_i \; V_{i\perp}']^\tp$ and, since the rank of the submatrix of $P_i^\tp D Q_i$ containing its last $m-r_i$ rows and $n-r_i$ columns is upper bounded by $\rank P_i^\tp D Q_i = \rank D \le \oshort{r}-r \le \oshort{r}-r_i$, $\eta_i \in \tancone{\R_{\le \oshort{r}}^{m \times n}}{X_i}$ according to \eqref{eq:TangentConeRealDeterminantalVarietyMatrix}.
\end{proof}

\begin{proposition}
\label{prop:MinimalInnerLimitTangentConeRealDeterminantalVariety}
For every $X \in \R_{\ushort{r}}$, if $\ushort{r} < r$ and $(X_i)_{i \in \N}$ is the sequence in $\R_r^{m \times n}$ obtained by applying Lemma~\ref{lemma:DenseTangentConeDecompositionDecreaseRank} to $X$, then
\begin{equation*}
\inlim_{i \to \infty} \tancone{\R_{\le \oshort{r}}^{m \times n}}{X_i} = \tancone{\R_{\le \oshort{r}-r+\ushort{r}}^{m \times n}}{X}.
\end{equation*}
\end{proposition}

\begin{proof}
It suffices to prove the inclusion $\subseteq$ thanks to Proposition~\ref{prop:LowerBoundInnerLimitTangentConeRealDeterminantalVariety}.
Let $(\eta_i)_{i \in \N}$ be a sequence converging to $\eta$ such that $\eta_i \in \tancone{\R_{\le \oshort{r}}^{m \times n}}{X_i}$ for every $i \in \N$. 
Choose $[\bar{U}_\perp \; U_\perp] \in \mathcal{U}_\perp$ and $[\bar{V}_\perp \; V_\perp] \in \mathcal{V}_\perp$ with the same block sizes as $[\bar{U}_i \; U_{i\perp}]$ and $[\bar{V}_i \; V_{i\perp}]$, respectively, and write
\begin{equation*}
\eta = [U \; \bar{U}_\perp \; U_\perp] \begin{bmatrix}
A & B & C \\ D & E & F \\ G & H & K
\end{bmatrix} [V \; \bar{V}_\perp \; V_\perp]^\tp.
\end{equation*}
By \eqref{eq:TangentConeRealDeterminantalVarietyMatrix}, it suffices to show that $\rank \left[ \begin{smallmatrix} E & F \\ H & K \end{smallmatrix} \right] \le \oshort{r}-r$.
Also by \eqref{eq:TangentConeRealDeterminantalVarietyMatrix}, for every $i \in \N$,
\begin{align*}
\eta_i = [U \; \bar{U}_i \; U_{i\perp}] \begin{bmatrix}
A_i & B_i & C_i \\ D_i & E_i & F_i \\ G_i & H_i & K_i
\end{bmatrix} [V \; \bar{V}_i \; V_{i\perp}]^\tp,&&
K_i = U_{i\perp}^\tp \eta_i V_{i\perp} \in \R_{\le \oshort{r}-r}^{m-r \times n-r}.
\end{align*}
In view of Lemma~\ref{lemma:DenseTangentConeDecompositionDecreaseRank}, for every $P \in \mathcal{O}_{m-\ushort{r}}$ and $Q \in \mathcal{O}_{n-\ushort{r}}$, $\left([\bar{U}_\perp \; U_\perp] \left[\begin{smallmatrix} P_{1,1} & P_{1,2} \\ P_{2,1} & P_{2,2} \end{smallmatrix}\right],\, [\bar{V}_\perp \; V_\perp] \left[\begin{smallmatrix} Q_{1,1} & Q_{1,2} \\ Q_{2,1} & Q_{2,2} \end{smallmatrix}\right]\right)$ is a cluster point of $(([\bar{U}_i \; U_{i\perp}],\, [\bar{V}_i \; V_{i\perp}]))_{i \in \N}$, hence
\begin{equation*}
\left(\begin{bmatrix} \bar{U}_\perp & U_\perp \end{bmatrix} \begin{bmatrix} P_{1,2} \\ P_{2,2} \end{bmatrix}\right)^\tp \eta \left(\begin{bmatrix} \bar{V}_\perp & V_\perp \end{bmatrix} \begin{bmatrix} Q_{1,2} \\ Q_{2,2} \end{bmatrix}\right)
= \begin{bmatrix} P_{1,2} \\ P_{2,2} \end{bmatrix}^\tp \begin{bmatrix} E & F \\ H & K \end{bmatrix} \begin{bmatrix} Q_{1,2} \\ Q_{2,2} \end{bmatrix}
\end{equation*}
is a cluster point of $(K_i)_{i \in \N}$ and has therefore a rank not larger than $\oshort{r}-r$. Thus, each $m-r \times n-r$ submatrix of $\left[ \begin{smallmatrix} E & F \\ H & K \end{smallmatrix} \right]$ has a rank not larger than $\oshort{r}-r$, which implies that $\rank \left[ \begin{smallmatrix} E & F \\ H & K \end{smallmatrix} \right] \le \oshort{r}-r$.
\end{proof}

\begin{proposition}
\label{prop:RelativeContinuityTangentConeRealDeterminantalVariety}
For every sequence $(X_i)_{i \in \N}$ in $\R_{\le r}^{m \times n}$ converging to $X \in \R_r^{m \times n}$,
\begin{equation*}
\setlim_{i \to \infty} \tancone{\R_{\le \oshort{r}}^{m \times n}}{X_i} = \tancone{\R_{\le \oshort{r}}^{m \times n}}{X}.
\end{equation*}
\end{proposition}

\begin{proof}
By \eqref{eq:DistanceToRealDeterminantalVariety}, every sequence in $\R_{\le r}^{m \times n}$ converging to a point in $\R_r^{m \times n}$ contains at most a finite number of elements in $\R_{< r}^{m \times n}$. Thus, it suffices to consider a sequence $(X_i)_{i \in \N}$ in $\R_r^{m \times n}$ converging to $X \in \R_r^{m \times n}$. In view of Proposition~\ref{prop:LowerBoundInnerLimitTangentConeRealDeterminantalVariety}, we only need to prove that
\begin{equation*}
\outlim_{i \to \infty} \tancone{\R_{\le \oshort{r}}^{m \times n}}{X_i} \subseteq \tancone{\R_{\le \oshort{r}}^{m \times n}}{X}.
\end{equation*}
Let $\eta_i \in \tancone{\R_{\le \oshort{r}}^{m \times n}}{X_i}$ for every $i \in \N$ and $(\eta_i)_{i \in \N}$ have $\eta \in \R^{m \times n}$ as cluster point. We need to prove that $\eta \in \tancone{\R_{\le \oshort{r}}^{m \times n}}{X}$.
Let $X = [U \; U_\perp ] \diag(\Sigma, 0_{m-r \times n-r}) [V \; V_\perp]^\tp$ be an SVD and let us use the notation of Lemma~\ref{lemma:ConvergenceTangentConeDecompositionConstantRank}. Then, by \eqref{eq:TangentConeRealDeterminantalVarietyMatrix}, for every $i \in \N$,
$
\eta_i = [U_i \; U_{i\perp}] \left[\begin{smallmatrix}
A_i & B_i \\ C_i & D_i
\end{smallmatrix}\right] [V_i \; V_{i\perp}]^\tp
$
with $A_i = U_i^\mpinv \eta_i V_i^{\mpinv\tp} \in \R^{r \times r}$, $B_i = U_i^\mpinv \eta_i V_{i\perp}^{\mpinv\tp} \in \R^{r \times n-r}$, $C_i = U_{i\perp}^\mpinv \eta_i V_i^{\mpinv\tp} \in \R^{m-r \times r}$ and $D_i = U_{i\perp}^\mpinv \eta_i V_{i\perp}^{\mpinv\tp} \in \R_{\le \oshort{r}-r}^{m-r \times n-r}$.
Let $(\eta_{i_k})_{k \in \N}$ be a subsequence of $(\eta_i)_{i \in \N}$ converging to $\eta$.
Then, for every $k \in \N$,
$
\eta_{i_k} = [U_{i_k} \; U_{i_k\perp}] \left[\begin{smallmatrix}
A_{i_k} & B_{i_k} \\ C_{i_k} & D_{i_k}
\end{smallmatrix}\right] [V_{i_k} \; V_{i_k\perp}]^\tp
$
and, since the subsequences $(A_{i_k})_{k \in \N}$, $(B_{i_k})_{k \in \N}$, $(C_{i_k})_{k \in \N}$ and $(D_{i_k})_{k \in \N}$ respectively converge to $A := U^\tp \eta V$, $B := U^\tp \eta V_\perp$, $C := U_\perp^\tp \eta V$ and $D := U_\perp^\tp \eta V_\perp$ with $\rank D \le \oshort{r}-r$ as $\R_{\le \oshort{r}-r}^{m-r \times n-r}$ is closed,
$
\eta = [U \; U_\perp] \left[\begin{smallmatrix}
A & B \\ C & D
\end{smallmatrix}\right] [V \; V_\perp]^\tp,
$
which shows that $\eta \in \tancone{\R_{\le \oshort{r}}^{m \times n}}{X}$ according to \eqref{eq:TangentConeRealDeterminantalVarietyMatrix}.
\end{proof}

\begin{proposition}
\label{prop:UpperBoundOuterLimitTangentConeRealDeterminantalVariety}
For every sequence $(X_i)_{i \in \N}$ in $\R_{\le r}^{m \times n}$ converging to $X \in \R_{\ushort{r}}^{m \times n}$,
\begin{equation*}
\outlim_{i \to \infty} \tancone{\R_{\le \oshort{r}}^{m \times n}}{X_i} \subseteq \tancone{\R_{\le \oshort{r}+r-\ushort{r}}^{m \times n}}{X}.
\end{equation*}
\end{proposition}

\begin{proof}
The case where $\ushort{r} = r$ has been considered in Proposition~\ref{prop:RelativeContinuityTangentConeRealDeterminantalVariety}. We focus here on the case where $\ushort{r} < r$. The result is trivial if $\oshort{r}+r-\ushort{r} \ge \min\{m,n\}$ since $\tancone{\R^{m \times n}}{X} = \R^{m \times n}$. We therefore assume that $\oshort{r}+r-\ushort{r} < \min\{m,n\}$.
Let $(X_i)_{i \in \N}$ be a sequence in $\R_{\le r}^{m \times n}$ converging to $X \in \R_{\ushort{r}}^{m \times n}$. Then, for every $i \in \N$ large enough, $\ushort{r} \le \rank X_i \le r$.
We have to prove the following inclusion: if $\eta_i \in \tancone{\R_{\le \oshort{r}}^{m \times n}}{X_i}$ for every $i \in \N$ and $(\eta_i)_{i \in \N}$ has $\eta \in \R^{m \times n}$ as cluster point, then $\eta \in \tancone{\R_{\le \oshort{r}+r-\ushort{r}}^{m \times n}}{X}$.
Let $(\eta_{j_k})_{k \in \N}$ be a subsequence of $(\eta_i)_{i \in \N}$ converging to $\eta$ such that $\rank X_{j_k} = \tilde{r} \in \{\ushort{r}, \dots, r\}$ for every $k \in \N$. We are going to prove that $\eta \in \tancone{\R_{\le \oshort{r}+\tilde{r}-\ushort{r} }^{m \times n}}{X}$; the result will then follow from the inclusion $\tancone{\R_{\le \oshort{r}+\tilde{r}-\ushort{r} }^{m \times n}}{X} \subseteq \tancone{\R_{\le \oshort{r}+r-\ushort{r} }^{m \times n}}{X}$. If $\tilde{r} = \ushort{r}$, the result follows from Proposition~\ref{prop:RelativeContinuityTangentConeRealDeterminantalVariety}. We therefore assume that $\tilde{r} > \ushort{r}$ and use the notation of Lemma~\ref{lemma:ConvergenceTangentConeDecompositionDecreaseRank} applied to $(X_{j_k})_{k \in \N}$, the $r$ in Lemma~\ref{lemma:ConvergenceTangentConeDecompositionDecreaseRank} being $\tilde{r}$ here. By \eqref{eq:TangentConeRealDeterminantalVarietyMatrix},
for every $k \in \N$, as $\eta_{i_k} \in \tancone{\R_{\le \oshort{r}}^{m \times n}}{X_{i_k}}$,
\begin{equation*}
\eta_{i_k} = [U_{i_k} \; \bar{U}_{i_k} \; U_{i_k\perp}] \begin{bmatrix}
A_{i_k} & B_{i_k} & C_{i_k} \\ D_{i_k} & E_{i_k} & F_{i_k} \\ G_{i_k} & H_{i_k} & K_{i_k}
\end{bmatrix} [V_{i_k} \; \bar{V}_{i_k} \; V_{i_k\perp}]^\tp
\end{equation*}
with
\begin{align*}
A_{i_k} &= U_{i_k}^\mpinv \eta_{i_k} V_{i_k}^{\mpinv\tp} \in \R^{\ushort{r} \times \ushort{r}},&&
B_{i_k} = U_{i_k}^\mpinv \eta_{i_k} \bar{V}_{i_k}^{\mpinv\tp} \in \R^{\ushort{r} \times \tilde{r}-\ushort{r}},&&
C_{i_k} = U_{i_k}^\mpinv \eta_{i_k} V_{i_k\perp}^{\mpinv\tp} \in \R^{\ushort{r} \times n-\tilde{r}},\\
D_{i_k} &= \bar{U}_{i_k}^\mpinv \eta_{i_k} V_{i_k}^{\mpinv\tp} \in \R^{\tilde{r}-\ushort{r} \times \ushort{r}},&&
E_{i_k} = \bar{U}_{i_k}^\mpinv \eta_{i_k} \bar{V}_{i_k}^{\mpinv\tp} \in \R^{\tilde{r}-\ushort{r} \times \tilde{r}-\ushort{r}},&&
F_{i_k} = \bar{U}_{i_k}^\mpinv \eta_{i_k} V_{i_k\perp}^{\mpinv\tp} \in \R^{\tilde{r}-\ushort{r} \times n-\tilde{r}},\\
G_{i_k} &= U_{i_k\perp}^\mpinv \eta_{i_k} V_{i_k}^{\mpinv\tp} \in \R^{m-\tilde{r} \times \ushort{r}},&&
H_{i_k} = U_{i_k\perp}^\mpinv \eta_{i_k} \bar{V}_{i_k}^{\mpinv\tp} \in \R^{m-\tilde{r} \times \tilde{r}-\ushort{r}},&&
K_{i_k} = U_{i_k\perp}^\mpinv \eta_{i_k} V_{i_k\perp}^{\mpinv\tp} \in \R_{\le \oshort{r}-\tilde{r}}^{m-\tilde{r} \times n-\tilde{r}},
\end{align*}
and respective limits
\begin{align*}
A &:= U^\tp \eta V \in \R^{\ushort{r} \times \ushort{r}},&&
B := U^\tp \eta \bar{V}_{\perp} \in \R^{\ushort{r} \times \tilde{r}-\ushort{r}},&&
C := U^\tp \eta V_\perp \in \R^{\ushort{r} \times n-\tilde{r}},\\
D &:= \bar{U}_{\perp}^\tp \eta V \in \R^{\tilde{r}-\ushort{r} \times \ushort{r}},&&
E := \bar{U}_{\perp}^\tp \eta \bar{V}_{\perp} \in \R^{\tilde{r}-\ushort{r} \times \tilde{r}-\ushort{r}},&&
F := \bar{U}_{\perp}^\tp \eta V_\perp \in \R^{\tilde{r}-\ushort{r} \times n-\tilde{r}},\\
G &:= U_\perp^\tp \eta V \in \R^{m-\tilde{r} \times \ushort{r}},&&
H := U_\perp^\tp \eta \bar{V}_\perp \in \R^{m-\tilde{r} \times \tilde{r}-\ushort{r}},&&
K := U_\perp^\tp \eta V_\perp \in \R_{\le \oshort{r}-\tilde{r}}^{m-\tilde{r} \times n-\tilde{r}}.
\end{align*}
Thus, taking the limit as $k \to \infty$ on both sides yields
\begin{equation*}
\eta = [U \; \bar{U}_\perp \; U_\perp] \begin{bmatrix}
A & B & C \\ D & E & F \\ G & H & K
\end{bmatrix} [V \; \bar{V}_\perp \; V_\perp]^\tp,
\end{equation*}
with, by Proposition~\ref{prop:BoundRank2*2BlockMatrices} applied with $k := \tilde{r}-\ushort{r}$, $p := m-\tilde{r}$, $q := n-\tilde{r}$ and $s := \oshort{r}-\tilde{r}$,
$
\rank \left[\begin{smallmatrix}
E & F \\ H & K
\end{smallmatrix}\right]
\le \tilde{r}-\ushort{r} + \min\{\oshort{r}-\ushort{r}, m-\tilde{r}, n-\tilde{r}\}
= \oshort{r}-\ushort{r} + \tilde{r}-\ushort{r},
$
which shows that $\eta \in \tancone{\R_{\le \oshort{r}+\tilde{r}-\ushort{r}}^{m \times n}}{X}$.
\end{proof}

\begin{proposition}
\label{prop:MaximalOuterLimitTangentConeRealDeterminantalVariety}
For every $X \in \R_{\ushort{r}}$, if $\ushort{r} < r$ and $(X_i)_{i \in \N}$ is the sequence in $\R_r^{m \times n}$ obtained by applying Lemma~\ref{lemma:DenseTangentConeDecompositionDecreaseRank} to $X$, then
\begin{equation*}
\outlim_{i \to \infty} \tancone{\R_{\le \oshort{r}}^{m \times n}}{X_i} = \tancone{\R_{\le \oshort{r}+r-\ushort{r}}^{m \times n}}{X}.
\end{equation*}
\end{proposition}

\begin{proof}
It suffices to prove the inclusion $\supseteq$ thanks to Proposition~\ref{prop:UpperBoundOuterLimitTangentConeRealDeterminantalVariety}.
For a given $\eta \in \tancone{\R_{\le \oshort{r}+r-\ushort{r}}^{m \times n}}{X}$, let us construct a sequence $(\eta_i)_{i\in\N}$ having $\eta$ as cluster point and such that $\eta_i \in \tancone{\R_{\le \oshort{r}}^{m \times n}}{X_i}$ for every $i \in \N$, and the proof will be complete. Let $[\bar{U}_\perp \; U_\perp] \in \mathcal{U}_\perp$ and $[\bar{V}_\perp \; V_\perp] \in \mathcal{V}_\perp$. By \eqref{eq:TangentConeRealDeterminantalVarietyMatrix},
\begin{equation*}
\eta = [U \; \bar{U}_\perp \; U_\perp] \begin{bmatrix}
A & B & C \\ D & E & F \\ G & H & K
\end{bmatrix} [V \; \bar{V}_\perp \; V_\perp]^\tp
\end{equation*}
with $A \in \R^{\ushort{r} \times \ushort{r}}$, $B \in \R^{\ushort{r} \times r-\ushort{r}}$, $C \in \R^{\ushort{r} \times n-r}$, $D \in \R^{r-\ushort{r} \times \ushort{r}}$, $E \in \R^{r-\ushort{r} \times r-\ushort{r}}$, $F \in \R^{r-\ushort{r} \times n-r}$, $G \in \R^{m-r \times \ushort{r}}$, $H \in \R^{m-r \times r-\ushort{r}}$, $K \in \R^{m-r \times n-r}$ and $\rank \left[\begin{smallmatrix} E & F \\ H & K \end{smallmatrix}\right] \le \oshort{r}-\ushort{r}+r-\ushort{r} = 2(r-\ushort{r})+(\oshort{r}-r)$. By Proposition~\ref{prop:PermutationRank2*2BlockMatrices} applied with $k := r-\ushort{r}$, $p := m-r$, $q := n-r$ and $s := \oshort{r}-r$, there exist $\left[\begin{smallmatrix} U_{1,1} & U_{1,2} \\ U_{2,1} & U_{2,2} \end{smallmatrix}\right] \in \mathcal{O}_{m-\ushort{r}}$ and $\left[\begin{smallmatrix} V_{1,1} & V_{1,2} \\ V_{2,1} & V_{2,2} \end{smallmatrix}\right] \in \mathcal{O}_{n-\ushort{r}}$ such that $\left[\begin{smallmatrix} E & F \\ H & K \end{smallmatrix}\right] = \left[\begin{smallmatrix} U_{1,1} & U_{1,2} \\ U_{2,1} & U_{2,2} \end{smallmatrix}\right] \left[\begin{smallmatrix} E' & F' \\ H' & K' \end{smallmatrix}\right] \left[\begin{smallmatrix} V_{1,1} & V_{1,2} \\ V_{2,1} & V_{2,2} \end{smallmatrix}\right]^\tp$ with $E' \in \R^{r-\ushort{r} \times r-\ushort{r}}$, $F' \in \R^{r-\ushort{r} \times n-r}$, $H' \in \R^{m-r \times r-\ushort{r}}$ and $K' \in \R_{\le \oshort{r}-r}^{m-r \times n-r}$. If $\left[\begin{smallmatrix} \bar{U}_\perp' & U_\perp' \end{smallmatrix}\right] := \left[\begin{smallmatrix} \bar{U}_\perp & U_\perp \end{smallmatrix}\right] \left[\begin{smallmatrix} U_{1,1} & U_{1,2} \\ U_{2,1} & U_{2,2} \end{smallmatrix}\right]$, $\left[\begin{smallmatrix} \bar{V}_\perp' & V_\perp' \end{smallmatrix}\right] := \left[\begin{smallmatrix} \bar{V}_\perp & V_\perp \end{smallmatrix}\right] \left[\begin{smallmatrix} V_{1,1} & V_{1,2} \\ V_{2,1} & V_{2,2} \end{smallmatrix}\right]$, $\left[\begin{smallmatrix} B' & C' \end{smallmatrix}\right] := \left[\begin{smallmatrix} B & C \end{smallmatrix}\right] \left[\begin{smallmatrix} V_{1,1} & V_{1,2} \\ V_{2,1} & V_{2,2} \end{smallmatrix}\right]$ and $\left[\begin{smallmatrix} D' \\ G' \end{smallmatrix}\right] := \left[\begin{smallmatrix} U_{1,1} & U_{1,2} \\ U_{2,1} & U_{2,2} \end{smallmatrix}\right]^\tp \left[\begin{smallmatrix} D \\ G \end{smallmatrix}\right]$, then
\begin{equation*}
\eta = [U \; \bar{U}_\perp' \; U_\perp'] \begin{bmatrix}
A & B' & C' \\ D' & E' & F' \\ G' & H' & K'
\end{bmatrix} [V \; \bar{V}_\perp' \; V_\perp']^\tp.
\end{equation*}
Thus, for every $i \in \N$,
\begin{equation*}
\eta_i := [U \; \bar{U}_i \; U_{i\perp}] \begin{bmatrix}
A & B' & C' \\ D' & E' & F' \\ G' & H' & K'
\end{bmatrix} [V \; \bar{V}_i \; V_{i\perp}]^\tp \in \tancone{\R_{\le \oshort{r}}^{m \times n}}{X_i}.
\end{equation*}
By construction, $\eta$ is a cluster point of $(\eta_i)_{i \in \N}$ and therefore belongs to $\outlim_{i\to\infty} \tancone{\R_{\le \oshort{r}}^{m \times n}}{X_i}$.
\end{proof}

\section{Complementary results}
\label{sec:ComplementaryResults}
In this section, $\ushort{r}$, $r$ and $\oshort{r}$ are still positive integers such that $\ushort{r} \le r \le \oshort{r} < \min\{m,n\}$ and we compute inner and outer limits relative to $\R_r^{m \times n}$ and $\R_{\le r}^{m \times n}$ at $X \in \R_{\ushort{r}}^{m \times n}$, now of the correspondences $\regtancone{\R_{\le \oshort{r}}^{m \times n}}{\cdot}$, $\regnorcone{\R_{\le \oshort{r}}^{m \times n}}{\cdot}$, $\norcone{\R_{\le \oshort{r}}^{m \times n}}{\cdot}$ and $\connorcone{\R_{\le \oshort{r}}^{m \times n}}{\cdot}$ reviewed in Sections~\ref{subsec:TangentNormalCones} and \ref{subsec:TangentNormalConesRealDeterminantalVariety}. More precisely, Section~\ref{subsec:InnerOuterLimitsRegularTangentConeRealDeterminantalVariety} focuses on $\regtancone{\R_{\le \oshort{r}}^{m \times n}}{\cdot}$ while the three others concern the normal cones. Section~\ref{subsec:RelativeContinuityNormalConesRealDeterminantalVariety} considers the simple case where $\ushort{r} = r$ whereas the two others deal with the general case where $\ushort{r} \le r$, inner and outer limits being respectively studied in Sections~\ref{subsec:InnerLimitsNormalConesRealDeterminantalVariety} and \ref{subsec:OuterLimitsNormalConesRealDeterminantalVariety}.

\subsection{Inner and outer limits of $\regtancone{\R_{\le \oshort{r}}^{m \times n}}{\cdot}$ relative to $\R_r^{m \times n}$ and $\R_{\le r}^{m \times n}$}
\label{subsec:InnerOuterLimitsRegularTangentConeRealDeterminantalVariety}
As mentioned after the proof of \cite[Theorem~6.26]{RockafellarWets}, $\regtancone{\mathcal{S}}{\cdot}$ is not inner semicontinuous for an arbitrary nonempty subset $\mathcal{S}$ of $\R^{m \times n}$. The following result shows that it is however the case if $\mathcal{S} = \R_{\le \oshort{r}}^{m \times n}$.

\begin{corollary}
\label{coro:InnerOuterLimitsRegularTangentConeRealDeterminantalVariety}
For every sequence $(X_i)_{i \in \N}$ in $\R_{\le r}^{m \times n}$ converging to $X \in \R_{\ushort{r}}^{m \times n}$,
\begin{equation}
\label{eq:LowerUpperBoundsInnerOuterLimitsRegularTangentConeRealDeterminantalVariety}
\regtancone{\R_{\le \oshort{r}}^{m \times n}}{X}
= \tancone{\R_{\ushort{r}}^{m \times n}}{X}
\subseteq \inlim_{i \to \infty} \regtancone{\R_{\le \oshort{r}}^{m \times n}}{X_i}
\subseteq \outlim_{i \to \infty} \regtancone{\R_{\le \oshort{r}}^{m \times n}}{X_i}
\subseteq \tancone{\R_{\le 2r-\ushort{r}}^{m \times n}}{X}.
\end{equation}
Moreover, for every $X \in \R_{\ushort{r}}^{m \times n}$, if $\ushort{r} < r$ and $(X_i)_{i \in \N}$ is the sequence in $\R_r^{m \times n}$ obtained by applying Lemma~\ref{lemma:DenseTangentConeDecompositionDecreaseRank} to $X$, then
\begin{equation}
\label{eq:MinimalMaximalInnerOuterLimitsRegularTangentConeRealDeterminantalVariety}
\tancone{\R_{\ushort{r}}^{m \times n}}{X}
= \inlim_{i \to \infty} \regtancone{\R_{\le \oshort{r}}^{m \times n}}{X_i}
\subsetneq \outlim_{i \to \infty} \regtancone{\R_{\le \oshort{r}}^{m \times n}}{X_i}
= \tancone{\R_{\le 2r-\ushort{r}}^{m \times n}}{X}.
\end{equation}
Thus, for every $X \in \R_{\ushort{r}}^{m \times n}$,
\begin{align*}
\inlim_{\R_{\le r}^{m \times n} \ni Z \to X} \regtancone{\R_{\le \oshort{r}}^{m \times n}}{Z}
&= \inlim_{\R_r^{m \times n} \ni Z \to X} \regtancone{\R_{\le \oshort{r}}^{m \times n}}{Z}
= \tancone{\R_{\ushort{r}}^{m \times n}}{X}
= \regtancone{\R_{\le \oshort{r}}^{m \times n}}{X},\\
\outlim_{\R_{\le r}^{m \times n} \ni Z \to X} \regtancone{\R_{\le \oshort{r}}^{m \times n}}{Z}
&= \outlim_{\R_r^{m \times n} \ni Z \to X} \regtancone{\R_{\le \oshort{r}}^{m \times n}}{Z}
= \tancone{\R_{\le 2r-\ushort{r}}^{m \times n}}{X}.
\end{align*}
In particular, the correspondence $\regtancone{\R_{\le \oshort{r}}^{m \times n}}{\cdot}$ is:
\begin{itemize}
\item inner semicontinuous;
\item non-outer-semicontinuous relative to $\R_r^{m \times n}$ at every $X \in \R_{< r}^{m \times n}$;
\item continuous relative to $\R_{\le r}^{m \times n}$ at every $X \in \R_r^{m \times n}$.
\end{itemize}
\end{corollary}

\begin{proof}
The first inclusion of \eqref{eq:LowerUpperBoundsInnerOuterLimitsRegularTangentConeRealDeterminantalVariety} follows, mutatis mutandis, from the argument used in the proof of Proposition~\ref{prop:LowerBoundInnerLimitTangentConeRealDeterminantalVariety}. In view of \eqref{eq:RegularTangentConeRealDeterminantalVariety}, the first equality of \eqref{eq:MinimalMaximalInnerOuterLimitsRegularTangentConeRealDeterminantalVariety}, the last inclusion of \eqref{eq:LowerUpperBoundsInnerOuterLimitsRegularTangentConeRealDeterminantalVariety} and the last equality of \eqref{eq:MinimalMaximalInnerOuterLimitsRegularTangentConeRealDeterminantalVariety} respectively follow from Propositions~\ref{prop:MinimalInnerLimitTangentConeRealDeterminantalVariety}, \ref{prop:UpperBoundOuterLimitTangentConeRealDeterminantalVariety} and \ref{prop:MaximalOuterLimitTangentConeRealDeterminantalVariety}.
\end{proof}

\subsection{Continuity of the normal cones to $\R_{\le \oshort{r}}^{m \times n}$ relative to $\R_{\le r}^{m \times n}$ on $\R_r^{m \times n}$}
\label{subsec:RelativeContinuityNormalConesRealDeterminantalVariety}
\begin{proposition}
\label{prop:RelativeContinuityNormalConesRealDeterminantalVariety}
The correspondences $\regnorcone{\R_{\le \oshort{r}}^{m \times n}}{\cdot}$, $\norcone{\R_{\le \oshort{r}}^{m \times n}}{\cdot}$ and $\connorcone{\R_{\le \oshort{r}}^{m \times n}}{\cdot}$ are continuous relative to $\R_{\le r}^{m \times n}$ at every $X \in \R_r^{m \times n}$.
\end{proposition}

\begin{proof}
The proof is based on Lemma~\ref{lemma:ConvergenceTangentConeDecompositionConstantRank} in a similar way as the proof of Proposition~\ref{prop:RelativeContinuityTangentConeRealDeterminantalVariety}.
\end{proof}

\subsection{Inner limits of the normal cones to $\R_{\le \oshort{r}}^{m \times n}$ relative to $\R_r^{m \times n}$ at points of $\R_{< r}^{m \times n}$}
\label{subsec:InnerLimitsNormalConesRealDeterminantalVariety}
We begin with a basic result based on \cite[Exercise~4.14, Proposition~4.15 and Corollary~11.35(b)]{RockafellarWets} and describing how the inner and outer limits interact with the polar for closed convex cones.
\begin{proposition}
\label{prop:InnerOuterLimitsPolarClosedConvexCones}
For every sequence $(\mathcal{S}_i)_{i \in \N}$ of closed convex cones in $\R^{m \times n}$,
\begin{align*}
\inlim_{i \to \infty} \mathcal{S}_i^- = \big(\outlim_{i \to \infty} \mathcal{S}_i\big)^-,&&
\outlim_{i \to \infty} \mathcal{S}_i^- \subseteq \big(\inlim_{i \to \infty} \mathcal{S}_i\big)^-.
\end{align*}
\end{proposition}

\begin{proof}
By \cite[Exercise~4.14 and Proposition~4.15]{RockafellarWets}, $\inlim\limits_{i \to \infty} \mathcal{S}_i$ is a closed convex cone and $\outlim\limits_{i \to \infty} \mathcal{S}_i$ is a closed cone not necessarily convex.
Thus, the inclusion $\outlim\limits_{i \to \infty} \mathcal{S}_i^- \subseteq \big(\inlim\limits_{i \to \infty} \mathcal{S}_i\big)^-$ follows from the implication $\implies$ in the third equivalence of \cite[Corollary~11.35(b)]{RockafellarWets}. Replacing $\mathcal{S}_i$ by $\mathcal{S}_i^-$ in that inclusion and taking the polar yields $\inlim\limits_{i \to \infty} \mathcal{S}_i^- \subseteq \big(\outlim\limits_{i \to \infty} \mathcal{S}_i\big)^-$. Furthermore, by the implication $\implies$ in the second equivalence of \cite[Corollary~11.35(b)]{RockafellarWets}, $\inlim\limits_{i \to \infty} \mathcal{S}_i^- \supseteq \big(\outlim\limits_{i \to \infty} \mathcal{S}_i\big)^{---} = \big(\outlim\limits_{i \to \infty} \mathcal{S}_i\big)^-$.
\end{proof}

We deduce the inner limit of the normal cones to $\R_{\le \oshort{r}}^{m \times n}$ relative to $\R_r^{m \times n}$ at every point of $\R_{< r}^{m \times n}$.

\begin{corollary}
For every $X \in \R_{\ushort{r}}^{m \times n}$, if $\ushort{r} < r$ and $(X_i)_{i \in \N}$ is the sequence in $\R_r^{m \times n}$ obtained by applying Lemma~\ref{lemma:DenseTangentConeDecompositionDecreaseRank} to $X$, then
\begin{equation*}
\inlim_{i \to \infty} \regnorcone{\R_{\le \oshort{r}}^{m \times n}}{X_i}
= \inlim_{i \to \infty} \norcone{\R_{\le \oshort{r}}^{m \times n}}{X_i}
= \inlim_{i \to \infty} \connorcone{\R_{\le \oshort{r}}^{m \times n}}{X_i}
= \{0_{m \times n}\}.
\end{equation*}
In particular, for every $X \in \R_{< r}^{m \times n}$,
\begin{equation*}
\inlim_{\R_r^{m \times n} \ni Z \to X} \regnorcone{\R_{\le \oshort{r}}^{m \times n}}{Z}
= \inlim_{\R_r^{m \times n} \ni Z \to X} \norcone{\R_{\le \oshort{r}}^{m \times n}}{Z}
= \inlim_{\R_r^{m \times n} \ni Z \to X} \connorcone{\R_{\le \oshort{r}}^{m \times n}}{Z}
= \{0_{m \times n}\}.
\end{equation*}
\end{corollary}

\begin{proof}
Observe that
\begin{align*}
\inlim_{i \to \infty} \regnorcone{\R_{\le \oshort{r}}^{m \times n}}{X_i}
&\subseteq \inlim_{i \to \infty} \norcone{\R_{\le \oshort{r}}^{m \times n}}{X_i}\\
&\subseteq \inlim_{i \to \infty} \connorcone{\R_{\le \oshort{r}}^{m \times n}}{X_i}\\
&= \big(\outlim_{i \to \infty} \regtancone{\R_{\le \oshort{r}}^{m \times n}}{X_i}\big)^-\\
&= \big(\tancone{\R_{\le 2r-\ushort{r}}^{m \times n}}{X}\big)^-\\
&= \regnorcone{\R_{\le 2r-\ushort{r}}^{m \times n}}{X}\\
&= \{0_{m \times n}\}.
\end{align*}
The two inclusions follow from \eqref{eq:InclusionsNormalCones}. The first equality follows from Proposition~\ref{prop:InnerOuterLimitsPolarClosedConvexCones}.
The second equality follows from \eqref{eq:MinimalMaximalInnerOuterLimitsRegularTangentConeRealDeterminantalVariety}. The last equality is clear if $2r-\ushort{r} \ge \min\{m,n\}$ and follows from \eqref{eq:RegularNormalConeRealDeterminantalVariety} otherwise since $\ushort{r} < 2r-\ushort{r}$.
\end{proof}

%
%

\subsection{Outer limits of the normal cones to $\R_{\le \oshort{r}}^{m \times n}$ relative to $\R_r^{m \times n}$ at points of $\R_{< r}^{m \times n}$}
\label{subsec:OuterLimitsNormalConesRealDeterminantalVariety}
As a direct consequence of \eqref{eq:RegularNormalConeRealDeterminantalVariety}, for every $X \in \R_{< \oshort{r}}^{m \times n}$,
\begin{equation*}
\outlim_{\R_{< \oshort{r}}^{m \times n} \ni Z \to X} \regnorcone{\R_{\le \oshort{r}}^{m \times n}}{Z}
= \{0_{m \times n}\}.
\end{equation*}
Otherwise, we have the following result concerning the regular normal cone.
\begin{proposition}
For every $X \in \R_r^{m \times n}$,
\begin{equation*}
\outlim_{\R_{\oshort{r}}^{m \times n} \ni Z \to X} \regnorcone{\R_{\le \oshort{r}}^{m \times n}}{Z}
= \outlim_{Z \to X} \regnorcone{\R_{\le \oshort{r}}^{m \times n}}{Z}
= \norcone{\R_{\le \oshort{r}}^{m \times n}}{X}.
\end{equation*}
\end{proposition}

\begin{proof}
The second equality is simply \eqref{eq:NormalCone}.
In the first equality, the inclusion $\subseteq$ follows from \eqref{eq:MonotonicityInnerOuterLimits}. Let us prove the converse inclusion. Let $(X_i)_{i \in \N}$ be a sequence in $\R_{\le \oshort{r}}^{m \times n}$ converging to $X \in \R_r^{m \times n}$ such that $\outlim_{i \to \infty} \regnorcone{\R_{\le \oshort{r}}^{m \times n}}{X_i} \ne \{0_{m \times n}\}$; such a sequence necessarily exists since $\norcone{\R_{\le \oshort{r}}^{m \times n}}{X} \ne \{0_{m \times n}\}$. Let $(X_{i_k})_{k \in \N}$ be the subsequence containing all elements of $(X_i)_{i \in \N}$ that have rank $\oshort{r}$; such a subsequence necessarily exists since $\outlim_{i \to \infty} \regnorcone{\R_{\le \oshort{r}}^{m \times n}}{X_i} \ne \{0_{m \times n}\}$. Let us establish that $\outlim_{k \to \infty} \regnorcone{\R_{\le \oshort{r}}^{m \times n}}{X_{i_k}} \supseteq \outlim_{i \to \infty} \regnorcone{\R_{\le \oshort{r}}^{m \times n}}{X_i}$ and the proof will be complete.
Let $\eta \in \outlim_{i \to \infty} \regnorcone{\R_{\le \oshort{r}}^{m \times n}}{X_i} \setminus \{0_{m \times n}\}$. Then, $\eta$ is a cluster point of a sequence $(\eta_i)_{i \in \N}$ such that $\eta_i \in \regnorcone{\R_{\le \oshort{r}}^{m \times n}}{X_i}$ for every $i \in \N$. Let $(\eta_{j_k})_{k \in \N}$ be a subsequence converging to $\eta$. Then, since $\eta \ne 0_{m \times n}$, $(X_{j_k})_{k \in \N}$ necessarily contains a subsequence of $(X_{i_k})_{k \in \N}$ and therefore $\eta \in \outlim_{k \to \infty} \regnorcone{\R_{\le \oshort{r}}^{m \times n}}{X_{i_k}}$.
\end{proof}

We already know from \eqref{eq:NormalConeOuterSemicontinuous} that the normal cone correspondence is outer semicontinuous. The following proposition gives a result a bit more precise.

\begin{proposition}
\label{prop:OuterLimitNormalConeRealDeterminantalVariety}
For every $X \in \R_{\ushort{r}}^{m \times n}$, if $\ushort{r} < r$ and $(X_i)_{i \in \N}$ is the sequence in $\R_r^{m \times n}$ obtained by applying Lemma~\ref{lemma:DenseTangentConeDecompositionDecreaseRank} to $X$, then
\begin{equation*}
\norcone{\R_{\le \oshort{r}}^{m \times n}}{X} \subseteq \outlim_{i \to \infty} \norcone{\R_{\le \oshort{r}}^{m \times n}}{X_i}.
\end{equation*}
In particular, for every $X \in \R_{\ushort{r}}^{m \times n}$,
\begin{equation*}
\outlim_{\R_r^{m \times n} \ni Z \to X} \norcone{\R_{\le \oshort{r}}^{m \times n}}{Z}
= \outlim_{\R_{\le r}^{m \times n} \ni Z \to X} \norcone{\R_{\le \oshort{r}}^{m \times n}}{Z}
= \outlim_{Z \to X} \norcone{\R_{\le \oshort{r}}^{m \times n}}{Z}
= \norcone{\R_{\le \oshort{r}}^{m \times n}}{X}.
\end{equation*}
\end{proposition}

\begin{proof}
For the first part, let $\eta \in \norcone{\R_{\le \oshort{r}}^{m \times n}}{X}$. There are $\tilde{U}_\perp \in \mathcal{U}_\perp$, $\tilde{V}_\perp \in \mathcal{V}_\perp$ and $A \in \R_{\le \min\{m,n\}-\oshort{r}}^{m-\ushort{r} \times n-\ushort{r}}$ such that $\eta = \tilde{U}_\perp A \tilde{V}_\perp^\tp$. There are $\tilde{U} \in \st(m-r,m-\ushort{r})$, $\tilde{V} \in \st(n-r,n-\ushort{r})$ and $A' \in \R_{\le \min\{m,n\}-\oshort{r}}^{m-r \times n-r}$ such that $A = \tilde{U} A' \tilde{V}^\tp$. Thus, $\eta = U_\perp A' V_\perp^\tp$ with $U_\perp := \tilde{U}_\perp \tilde{U} \in \st(m-r,m)$ and $V_\perp := \tilde{V}_\perp \tilde{V} \in \st(n-r,n)$. Since $\im U_\perp \subseteq (\im U)^\perp$ and $\im V_\perp \subseteq (\im V)^\perp$, $(U_\perp, V_\perp)$ is a cluster point of $((U_{i\perp}, V_{i\perp}))_{i \in \N}$. Thus, for every $i \in \N$, $\eta_i := U_{i\perp} A' V_{i\perp}^\tp \in \norcone{\R_{\le \oshort{r}}^{m \times n}}{X_i}$ and $\eta$ is a cluster point of $(\eta_i)_{i \in \N}$.

For the second part, observe that the last equality is simply \eqref{eq:NormalConeOuterSemicontinuous}. Thus, in view of \eqref{eq:MonotonicityInnerOuterLimits}, it suffices to prove that
\begin{equation*}
\norcone{\R_{\le \oshort{r}}^{m \times n}}{X} \subseteq \outlim_{\R_r^{m \times n} \ni Z \to X} \norcone{\R_{\le \oshort{r}}^{m \times n}}{Z}.
\end{equation*}
If $\ushort{r} = r$, the result follows from Proposition~\ref{prop:RelativeContinuityNormalConesRealDeterminantalVariety}. If $\ushort{r} < r$, the result follows from \eqref{eq:OuterLimitCorrespondence} and the first part.
\end{proof}

The following proposition shows that the outer limit of the correspondence $\connorcone{\R_{\le \oshort{r}}^{m \times n}}{\cdot}$ at points of $\R_{< r}^{m \times n}$ is different when it is considered relative to $\R_r^{m \times n}$ and to $\R_{\le r}^{m \times n}$.

\begin{proposition}
For every $X \in \R_{\ushort{r}}^{m \times n}$,
\begin{equation*}
\outlim_{\R_r^{m \times n} \ni Z \to X} \connorcone{\R_{\le \oshort{r}}^{m \times n}}{Z}
= \norcone{\R_{\le r}^{m \times n}}{X}
\subseteq \outlim_{\R_{\le r}^{m \times n} \ni Z \to X} \connorcone{\R_{\le \oshort{r}}^{m \times n}}{Z}
= \connorcone{\R_{\le \oshort{r}}^{m \times n}}{X}.
\end{equation*}
\end{proposition}

\begin{proof}
Let $X \in \R_{\ushort{r}}^{m \times n}$.
The first equality follows from \eqref{eq:ConvexifiedNormalConeRealDeterminantalVariety} and Proposition~\ref{prop:OuterLimitNormalConeRealDeterminantalVariety}:
\begin{equation*}
\outlim_{\R_r^{m \times n} \ni Z \to X} \connorcone{\R_{\le \oshort{r}}^{m \times n}}{Z}
= \outlim_{\R_r^{m \times n} \ni Z \to X} \norcone{\R_{\le r}^{m \times n}}{Z}
= \norcone{\R_{\le r}^{m \times n}}{X}.
\end{equation*}
Let us prove the other equality. On the one hand, by \eqref{eq:NonDeletedInnerOuterLimits},
\begin{equation*}
\outlim_{\R_{\le r}^{m \times n} \ni Z \to X} \connorcone{\R_{\le \oshort{r}}^{m \times n}}{Z}
\supseteq \connorcone{\R_{\le \oshort{r}}^{m \times n}}{X}
= \norcone{\R_{\ushort{r}}^{m \times n}}{X}.
\end{equation*}
On the other hand, by \eqref{eq:LowerUpperBoundsInnerOuterLimitsRegularTangentConeRealDeterminantalVariety}, for every sequence $(X_i)_{i \in \N}$ in $\R_{\le r}^{m \times n}$ converging to $X$,
\begin{equation*}
\tancone{\R_{\ushort{r}}^{m \times n}}{X}
\subseteq \inlim_{i \to \infty} \regtancone{\R_{\le \oshort{r}}^{m \times n}}{X_i}
\end{equation*}
and therefore, by Proposition~\ref{prop:InnerOuterLimitsPolarClosedConvexCones},
\begin{equation*}
\outlim_{i \to \infty} \connorcone{\R_{\le \oshort{r}}^{m \times n}}{X_i}
\subseteq \big(\inlim_{i \to \infty} \regtancone{\R_{\le \oshort{r}}^{m \times n}}{X_i}\big)^-
\subseteq (\tancone{\R_{\ushort{r}}^{m \times n}}{X})^-
= \norcone{\R_{\ushort{r}}^{m \times n}}{X}
= \connorcone{\R_{\le \oshort{r}}^{m \times n}}{X}.
\end{equation*}
Thus, $\outlim_{\R_{\le r}^{m \times n} \ni Z \to X} \connorcone{\R_{\le \oshort{r}}^{m \times n}}{Z}
\subseteq \norcone{\R_{\ushort{r}}^{m \times n}}{X}$.
\end{proof}

\section{Connection with the $a$-regularity of the Whitney stratification}
\label{sec:ConnectionRegularityWhitneyStratification}
In this section, we show that the $a$-regularity of the well-known Whitney stratification of the determinantal variety is included in Theorem~\ref{theorem:MainResult} as a particular case. Throughout the section, $\ushort{r}$, $r$ and $\oshort{r}$ are positive integers such that $\ushort{r} < r \le \oshort{r} < \min\{m,n\}$.
As mentioned in \cite[\S 4.1]{HosseiniUschmajew2019}, being a real algebraic variety, $\R_{\le \oshort{r}}^{m \times n}$ admits a Whitney stratification and, in particular, satisfies the so-called $a$-regularity condition introduced in \cite[\S 19]{Whitney1965}: for every sequence $(X_i)_{i \in \N}$ in $\R_r^{m \times n}$ converging to $X \in \R_{\ushort{r}}^{m \times n}$, if $\big(\tancone{\R_r^{m \times n}}{X_i}\big)_{i \in \N}$ converges to $\mathcal{T}$ in $\grass(\dim \R_r^{m \times n}, mn)$, then
\begin{equation}
\label{eq:RegularityWhitneyStratification}
\tancone{\R_{\ushort{r}}^{m \times n}}{X} \subseteq \mathcal{T}.
\end{equation}
We should make two remarks here. First, $\R^{m \times n}$ has been identified with $\R^{mn}$ and each tangent space to $\R_r^{m \times n}$ is thus seen as a linear subspace of $\R^{mn}$, i.e., as an element of $\grass(\dim \R_r^{m \times n}, mn)$. Secondly, the distance defined in \cite[(2.14)]{Whitney1965} is the gap distance according to \cite[(I-2-1) and (I-2-3)(1)]{FerrerGarciaPuerta1994} and \cite[(2.2)]{Whitney1965}, a distance known to induce the usual topology of the Grassmann manifold \cite[(I-2-6)]{FerrerGarciaPuerta1994}.

In this section, we show that the first inclusion of \eqref{eq:LowerUpperBoundsInnerOuterLimitsTangentConeRealDeterminantalVariety} reduces to \eqref{eq:RegularityWhitneyStratification} when $r = \oshort{r}$. This will follow from Proposition~\ref{prop:GrassmannPainlevéConvergences} that states that, in the Grassmann manifold, the convergence for the gap distance implies the convergence in the sense of Painlevé. Before introducing that result, we review the definition of the gap distance for the convenience of the reader.

First, let us recall from \cite[(2.1)]{BendokatZimmermannAbsil2020} that, given a positive integer $p \le n$, $\grass(p,n)$ is the smooth manifold of all $p$-dimensional linear subspaces of $\R^n$ and that every $\mathcal{G} \in \grass(p,n)$ can be identified with a unique orthogonal projection $G \in \R_p^{n \times n}$ called the orthogonal projection onto $\mathcal{G}$.
The gap distance between $\mathcal{G}_1$ and $\mathcal{G}_2$ in $\grass(p,n)$ is defined as $\norm{G_1-G_2}_2$ where $G_i$ is the orthogonal projection onto $\mathcal{G}_i$ for every $i \in \{1,2\}$ and $\norm{\cdot}_2$ is the spectral norm on $\R^{m \times n}$ \cite[(I-2-1)]{FerrerGarciaPuerta1994}. Let us mention that other topologically equivalent distances on the Grassmann manifold are given in \cite[(1) and Table~2]{YeLim2016}.

\begin{proposition}
\label{prop:GrassmannPainlevéConvergences}
If $(\mathcal{S}_i)_{i \in \N}$ converges to $\mathcal{S}$ in $\grass(p,n)$ endowed with the gap distance, then $(\mathcal{S}_i)_{i \in \N}$ converges to $\mathcal{S}$ in the sense of Painlevé.
\end{proposition}

\begin{proof}
For every $i \in \N$, let $P_i$ and $P$ denote the orthogonal projections onto $\mathcal{S}_i$ and $\mathcal{S}$, respectively. By hypothesis, $\lim_{i \to \infty} \norm{P_i-P}_2 = 0$. Let us prove that
\begin{equation*}
\outlim_{i \to \infty} \mathcal{S}_i \subseteq \mathcal{S} \subseteq \inlim_{i \to \infty} \mathcal{S}_i.
\end{equation*}
The first inclusion follows from the fact that, for every sequence $(v_i)_{i \in \N}$ such that $v_i \in \mathcal{S}_i$ for every $i \in \N$, since $P_i v_i = v_i$ for every $i \in \N$, each cluster point $v$ of $(v_i)_{i \in \N}$ satisfies $Pv = v$, i.e., belongs to $\mathcal{S}$. The second inclusion follows from the fact that, for every $v \in \mathcal{S}$, if $v_i := P_i v$ for every $i \in \N$, then $v_i \in \mathcal{S}_i$ for every $i \in \N$ and $(v_i)_{i \in \N}$ converges to $v$.
\end{proof}

Let us now prove that \eqref{eq:RegularityWhitneyStratification} follows from \eqref{eq:LowerUpperBoundsInnerOuterLimitsTangentConeRealDeterminantalVariety}. Let $(X_i)_{i \in \N}$ be a sequence in $\R_r^{m \times n}$ converging to $X \in \R_{\ushort{r}}^{m \times n}$ such that $\big(\tancone{\R_r^{m \times n}}{X_i}\big)_{i \in \N}$ converges to $\mathcal{T}$ in $\grass(\dim \R_r^{m \times n}, mn)$ endowed with the gap distance. Then, by Proposition~\ref{prop:GrassmannPainlevéConvergences}, $\mathcal{T} = \setlim_{i \to \infty} \tancone{\R_r^{m \times n}}{X_i}$. Thus, the first inclusion of \eqref{eq:LowerUpperBoundsInnerOuterLimitsTangentConeRealDeterminantalVariety} with $r = \oshort{r}$ reduces to \eqref{eq:RegularityWhitneyStratification}, as announced.

\bibliographystyle{alphaurl}
\bibliography{golikier_bib}
\end{document}